\documentclass[11pt,reqno]{article}
\usepackage{amsmath}
\usepackage{amsthm}
\usepackage{amssymb}
\usepackage{epsfig}
\usepackage{color}

\usepackage{calc}
\usepackage{xspace}
\usepackage{comment}
\usepackage{paralist}
\usepackage{url}
\usepackage{setspace}
\usepackage{rotating}
\usepackage[authoryear]{natbib}	
\usepackage{amsmath,amssymb}
\usepackage{array,arydshln}	
\usepackage{tabularx}	

\usepackage{threeparttable}
\usepackage{color}	
\usepackage{authblk}
\usepackage{hyperref}
\usepackage{enumitem}
\usepackage{multirow, booktabs, colortbl}

\usepackage{tikz}
\usetikzlibrary{positioning}
\usetikzlibrary{calendar}
\usetikzlibrary{decorations}
\usetikzlibrary{decorations.pathmorphing}
\usetikzlibrary{decorations.pathreplacing}
\usetikzlibrary{decorations.shapes}
\usetikzlibrary{decorations.text}

\newtheorem{thm}{Theorem}[section]
\newtheorem{lm}{Lemma}[section]

\newcommand {\ud}{\mathrm{d}}

\newcommand {\E} {E}
\newcommand {\A} {\alpha}
\newcommand {\pa} {\partial }

\newcommand{\les}{\lesssim}

\newcommand\argmin[1]{\underset{#1}{\operatorname{argmin}}}

\newcommand\ie{i.e.,\xspace}

\newcommand\OU{OU$_{no.chg}$ process}
\newcommand\OUchg{OU$_{chg}$ process}

\newcolumntype{Y}{>{\small\raggedright\arraybackslash}X}

\renewcommand{\arraystretch}{1.2}

\hypersetup{colorlinks=true,
linkcolor=blue,
urlcolor=blue}

\title{Robust test for dispersion parameter change\\ in discretely observed diffusion processes}
\author{Junmo Song\thanks{Department of Statistics, Kyungpook National University, 80 Daehakro, Bukgu, Daegu, 41566, Korea. Email:
jsong@knu.ac.kr}}
\affil{Department of Statistics, Kyungpook National University}
\date{}

\begin{document}
\maketitle

\begin{abstract}
This paper deals with the problem of testing for dispersion parameter change in discretely
observed diffusion processes when the observations are contaminated by outliers. To lessen the impact of outliers, we first calculate residuals using a robust estimate and then propose a trimmed-residual based CUSUM test. The proposed test is shown to converge weakly to a function of the Brownian bridge under the null hypothesis of no parameter change.  We conduct simulations to evaluate performances of the proposed test in the presence of outliers. Numerical results confirm that the proposed test posses a strong robust property against outliers. In real data analysis, we fit the Ornstein-Uhlenbeck process to KOSPI200 volatility index data and locate some change points that are not detected by a naive CUSUM test.
\end{abstract}

\section{Introduction}
\label{sec:intro}
Diffusion processes are usually expressed as solutions to stochastic differential equations (SDEs). Since SDEs are useful in describing stochastic phenomena, diffusion processes have long been popular in various fields.
In the field of finance, for example, the processes have been widely used to model the prices of underlying assets and instantaneous interest rates. Naturally, the need for statistical inference on diffusion processes has increased. In particular, estimation of discretely observed diffusion processes has attracted  much attention.  See  Dacunha-Castelle and Florens-Zmirou (1986), Kessler (1997), A\"{\i}t-Sahalia (2002), and Beskos et al (2009). Statistical testings such as parameter change test and specification test have also investigated by some authors. See, for example,  Iacus and Yoshida (2012) and Chen et al. (2008).

In this study, we are concerned with change point problem in diffusion processes.  It is well known that ignoring changes can lead to false  inference. Hence, change point problem has received a great deal of attention from researchers and practitioners. See the recent review paper by  Horváth and Rice (2015). For diffusion processes, Gregorio and Iacus (2008), Song and Lee (2009), Lee (2011), and Iacus and Yoshida (2012) investigated the problem of testing for dispersion parameter constancy in discretely observed diffusion processes.
%Note that the dispersion parameter is closely related to the volatility of underlying assets in financial applications. Since the volatility is a key factor in derivative pricing,
Since  the dispersion parameter is closely related to the volatility of underlying assets and  the volatility plays a crucial role in pricing financial derivatives, the exact inference on the dispersion parameter is  particularly important in financial applications.  In the cases where a continuous observation is assumed to be obtained, detection of drift parameter change is typically considered because dispersion coefficient can be exactly estimated in this framework. See, for example, Negri and Nishiyama (2012) and Tsukuda (2017).  %Recent work of Negri and Nishiyama (2017) is also noteworthy in that a test for both parameter change was proposed under the situation of discrete time observation.
%In many diffusion models used in pricing theory, the drift parameter is replaced with risk-free interest rate, the dispersion parameter, which is closely related to
%the volatility of financial assets, plays a critical role. In this regard, the inference on dispersion parameter is important.
%Note that the dispersion parameter is closely related to the volatility of underlying assets in financial applications. Since the volatility is a key factor in derivative pricing, the inference on dispersion parameter is important particularly in finance.

This paper focuses on the problem of detecting the dispersion parameter change, particularly when  a data set includes deviating observations. In the literature, deviating observations are commonly treated as jumps or outliers.  In the former cases, stochastic models  with  jump terms, usually  induced by Poisson processes, have been proposed to describe spiky observations.  See, for example, Kou (2002). In the latter cases, on the other hand, various robust methods for reducing the effect of outliers have been developed. For an overview on this area, we refer the reader to Maronna et al. (2006). In this study, we deal with the outlying observations from the latter point of view.

As is widely recognized, statistical inference such as estimation and testing are unduly influenced by outliers. Recently, Lee and Song (2013) and Song (2017) addressed that estimation of diffusion processes tends to be severely damaged by a small portion of outliers, particularly when sampling interval is short, as in high-frequency sampling cases. This is largely due to the fact that the transition distribution of the diffusion process approaches Gaussian distribution as the sampling interval gets shorter. It should be noted that  tests constructed using an estimator sensitive to outliers are likely to lead to false conclusions.
Furthermore,  such events that can cause deviating observations or parameter changes in fitted model are often observed in actual practice. In finance, changes of monetary policy and critical social events can be examples.  When outlying observations are included in a data set being suspected of  having parameter changes, it is not easy to determine whether the testing results are due to genuine changes or not. These technical and empirical reasons motivate us to consider the problem.

%Since Huber(1964), various methods that are robust to outliers have been developed. Among them, divergence-based methods have attracted increasing interest in the past decade. More specifically, estimators induced from various divergences such as density power (DP) divergence, $\gamma$-divergence, and S-divergence, are shown to have good robustness property with high efficiency. See Basu et al. (1998), Fujisawa and Eguchi (2008), and Ghosh and Basu (2017) for DP, $\gamma$-, and S-divergences, respectively.  Some robust tests have also been introduced using these estimators. For example, based on minimum DP divergence  estimator (MDPDE),  Basu et al. (2016) and Ghosh et al. (2016) proposed Wald-type tests and Kim (2018) studied the maximum entropy test in GARCH models. Kang and Song (2015) proposed MDPDE-based CUSUM test for parameter change in Poisson autoregressive models. These tests are found to inherit the robust property of each MDPDE.

The objective of  this  paper is to propose a parameter change test that is robust against outliers.  We introduce  a very intuitive and easy-to-implement test procedure: to lessen the impact of outliers on the procedure, (i) we first calculate residuals using a robust estimate and truncate the squares of the obtained residuals; (ii) then, we construct a CUSUM test using the trimmed ones. As a robust estimator, we employ the minimum density power divergence estimator (MDPDE) for diffusion processes introduced by Lee and Song (2013). Our simulation study below shows that the proposed test has strong robustness against outliers, whereas a naive CUSUM test without  any robust procedure is seriously compromised by outliers. Further, our real data application demonstrates that  analysis incorporating the proposed test can improve forecasting performance.

The rest of the paper is organized as follows. In Section 2, we introduce the residual-based CUSUM test and the MDPDE for diffusion processes. Then, we propose a robust CUSUM test for parameter change and derive its asymptotic null distribution. In Section 3, we conduct a simulation study to investigate the finite sample performance. Section 4 illustrates a real data application to KOSPI200 volatility index. Section 5 concludes and technical proofs are given in Section 6.

\section{Main Result}
Let us consider the following time-homogeneous diffusion process $\{X_t |  t \geq 0\}$ defined by
\begin{eqnarray}\label{DP}
dX_t = a(X_t,\theta) dt+\sigma dW_t,~~~X_0=x_0,
\end{eqnarray}
where $(\theta, \sigma) \in \mathbb{R}^p \times \mathbb{R}^+$ is unknown parameter and $\{ W_t | t \geq 0 \}$ denotes the
standard Wiener process. The real valued function $a$ is assumed to be known apart from $\theta$ and smooth
enough to admit a unique solution. We assume that a sample $\{X_{t_i} | 0\leq i \leq n\}$ is discretely observed, where $t_i=ih_n$ and $\{h_n\}$ is a sequence of positive numbers with $h_n \rightarrow 0$ and $nh_n \rightarrow \infty$.
It is noteworthy that the diffusion processes of the form
$dX_t =a(X_t,\theta)dt + \sigma\, b(X_t)dW_t$ can be reduced to (\ref{DP}) by using the Lamperti transformation and Ito's lemma.

\subsection{Naive CUSUM test for diffusion processes}
Based on the discrete observations, we now wish to test the following hypotheses:
\begin{eqnarray*}
&& H_0:\ \sigma\textrm{~~does not change over~} 0\leq t \leq nh_n\,.\quad \textrm{vs}.\quad H_1:\ \textrm{not}\ \ H_0\,.
\end{eqnarray*}
For this, we employ the CUSUM of squares test based on residuals as in Lee (2011). Residuals for the diffusion process (\ref{DP}) are defined as follows:
\begin{eqnarray}\label{res}
\hat{Z}_i:=\frac{X_{t_i}-X_{t_{i-1}}-a(X_{t_{i-1}}, \hat{\theta}_n)h_n}{\sqrt{h_n}\hat{\sigma}_n},\quad i=1,\cdots,n,
\end{eqnarray}
where $(\hat\theta_n,\hat\sigma_n)$ is an estimate of  $(\theta,\sigma)$. The above is deduced from the following Euler approximation  of (\ref{DP})
\begin{eqnarray}\label{Euler}
 X_{t_i} \approx X_{t_{i-1}}+a(X_{t_{i-1}},\theta)h_n +\sigma\sqrt{h_n}\frac{W_{t_i}-W_{t_{i-1}}}{\sqrt{h_n}}.
\end{eqnarray}
Using the residuals, we first introduce the CUSUM of squares statistics:
\begin{eqnarray*}
T_n:=\frac{1}{\sqrt{n}\hat{\tau}_n} \max_{1\leq k\leq n } \Big|\sum_{i=1}^k \hat{Z}^2_{i} -\frac{k}{n}\sum_{i=1}^n\hat{Z}_{i}^2\Big|
\end{eqnarray*}
where $\displaystyle \hat{\tau}_n^2=\frac{1}{n}\sum_{i=1}^n\hat{Z}_{i}^4-\Big(\frac{1}{n}\sum_{i=1}^n\hat{Z}_{i}^2\Big)^2$.
%Under the null hypothesis,  when the sample size $n$  is large the residuals will behave like i.i.d. random variables from $N(0,1)$, $T_n$

In order to establish the limiting null distribution of $T_n$, the following regularity conditions are required. We assume that the true parameter $(\theta_0, \sigma_0)$ belongs to the parameter space $\Theta$, which is a bounded subset of $\mathbb{R}^p\times[c,\infty)$ for some $c>0$.

\begin{enumerate}
\item[\bf A0.] The estimator $(\hat{\theta}_n, \hat{\sigma}_n)$ satisfies that $\sqrt{nh_n}||\hat{\theta}_n-\theta_0||=O_P(1)$ and
  $\sqrt{n}|\hat{\sigma}_n-\sigma_0|=O_P(1)$.
\item[\bf A1.]There exists a constant $C>0$ such that $|a(x,\theta_0)-a(y,\theta_0)| \leq C|x-y|$ for any $x, y \in \mathbb{R}$.
\item[\bf A2.]The process $X$  from (\ref{DP}) is ergodic with its invariant measure $\mu_0$ such that $\int
x^k\ud\mu_0(x)<\infty$\;for all\;$k\geq 0$.
\item[\bf A3.]$\sup_t E |X_t|^k < \infty$ for all $k\geq0$.
\item[\bf A4.]The function $a$ is continuously differentiable with respect to $x$
for all $\theta$ and the derivatives belong to $\mathcal{P}:=\{f(x, \theta)\big|\ |f| \leq C( 1+|x|^C)\ \textrm{for some }\ C>0\},$ where $C$ does not depend on the parameter.
\end{enumerate}

\noindent Then, from Lemma \ref{L3} with $S_n=\mathbb{R}$, we can obtain the following result.
\begin{thm}\label{T1}
Assume that {\bf A0} -- {\bf A4} hold. If $nh_n^2\rightarrow 0$, then under $H_0$,
\begin{eqnarray*}
T_n
\stackrel{d}{\rightarrow}\sup_{0\leq t \leq 1} |W^o_t|\quad as\quad n\rightarrow \infty,
\end{eqnarray*}
where $\{W_t^o | t\geq 0\}$ denotes a standard Brownian bridge.
\end{thm}

%\noindent Indeed, the residual-based CUSUM test has been commonly used for testing variance or parameter change in time series models. For example, Lee et al.(2003, 2004) construct the test for non stationary AR models and GARCH(1,1) models, respectively. Lee et al. (2010) also consider the test for (\ref{1}) based on the differently-defined residuals and derived its null distribution.\\

\subsection{ Robust CUSUM test for diffusion processes}
Now, we consider the situation where the observations are contaminated by outliers. It is well known that estimators using a Gaussian quasi-likelihood are strongly influenced by outliers. Also, it should be recalled that many estimation methods for diffusion
processes employ the maximum likelihood (ML) technique based on an approximated transition density; see, for example, Li (2013) and the papers therein.
 As aforementioned in the Introduction, since the transition distributions of the diffusion processes get close to the normal distribution as the sampling interval approaches zero, such estimators using ML methods are likely to produce biased estimates in the presence of outliers. Therefore, it can be naturally surmised that the residuals calculated from those biased estimates will not behave like ideal residuals, that is, i.i.d. random variables, and subsequently lead to a distortion of $T_n$.

In order to remedy the problem,  a robust estimator is first used to lessen the effect of outliers on parameter estimation. Next, note that  even if the parameters are properly estimated by a robust estimation method, the residuals corresponding to outliers still deviate from normal range. Thus, they need to be  truncated to prevent damaging the test procedure. In this study, we employ the MDPDE as a robust estimator. Further, in order to avoid some technical problems in the proofs, we use the trimmed ones of the squared residuals to construct test statistics in stead of using the squares of the truncated residuals.

Lee and Song (2013) introduced  a robust estimator for diffusion processes (\ref{DP}) using the density power divergence by Basu et al. (1998), and demonstrated in the simulation study that the estimator has a strong robust property with little loss
in asymptotic efficiency relative to the ML estimator (MLE). The MDPDE for (\ref{DP}) is defined as
\begin{eqnarray}\label{MDPDE}
(\hat{\theta}_n^\A, \hat{\sigma}_n^\A) = \argmin{(\theta,\sigma)\in\Theta}
\,\,\frac{1}{n}\sum_{i=1}^n H_i^\A(\theta,\sigma)\,,
 \end{eqnarray}
where {\setlength\arraycolsep{2pt}
\begin{eqnarray*}
&&H_i^\A(\theta,\sigma)\\
&&=\!\left\{\begin{array}{ll}
\displaystyle\frac{1}{\sigma^\A}
\Big[\frac{1}{\sqrt{1+\alpha}}- \Big(1+\frac{1}{\alpha}\Big)\,\exp\Big\{-\frac{\alpha}{2\sigma^2h_n} (X_{t_i}-X_{t_{i-1}}-a(X_{t_{i-1}},\theta)h_n)^2\Big\}\Big]&,\A>0\,,\\
 \\
\displaystyle \frac{1}{\sigma^2h_n}\big(X_{t_i}-X_{t_{i-1}}-a(X_{t_{i-1}},\theta)h_n\big)^2+\log \sigma^2&, \A=0\,.
 \end{array}\right.
 \end{eqnarray*}}
\noindent Here, the tuning parameter $\alpha$ controls the
trade-off between robustness and efficiency in the estimation procedure.  Note that the estimator with $\A=0$ becomes the Gaussian quasi-MLE. For more details on the MDPDE and its properties, see Basu et al. (1998).

We consider the following functions to truncate the residuals: for a given positive number $M$ and $x \geq 0$,
\begin{eqnarray*}
f_{1,M}(x)&=&x\ 1_{[0,M]}(x) + M\ 1_{(M,\infty]}(x), \\
f_{2,M}(x)&=&x\ 1_{[0,M]}(x)+ (2M-x)\ 1_{(M,2M]}(x),
\end{eqnarray*}
where  $1_A(x)$ is the indicator function of the set $A$. Using the trimming function, we now propose a
CUSUM test as follows: for $j=1,2$,
\begin{eqnarray*}
T_{j,n}^\A:=\frac{1}{\sqrt{n}\hat{\tau}_{j,n}} \max_{1\leq k\leq n } \Big|\sum_{i=1}^k f_{j,M}(\hat{Z}^2_{\A,i}) -\frac{k}{n}\sum_{i=1}^nf_{j,M}(\hat{Z}^2_{\A,i})\Big|,
\end{eqnarray*}
where $\hat{Z}_{\A,i}$'s are the ones calculated from (\ref{res}) using the MDPD estimate $(\hat{\theta}_n^\A, \hat{\sigma}_n^\A)$ and $\displaystyle\hat{\tau}_{j,n}^2=\frac{1}{n}\sum_{i=1}^nf_{j,M}^2(\hat{Z}_{\A,i}^2)-\Big(\frac{1}{n}\sum_{i=1}^nf_{j,M}(\hat{Z}_{\A,i}^2)\Big)^2$.

To derive the asymptotic null distribution of $T_{j,n}^\A$, $\sqrt{nh_n}||\hat{\theta}_n^\A-\theta_0||=O_P(1)$ and $\sqrt{n}|\hat{\sigma}_n^\A-\sigma_0|=O_P(1)$ are required.
For this, we additionally assume the following conditions to ensure the stochastic boundedness (cf. see Lee and Song (2013)).
\begin{enumerate}
\item[\bf A5.] $\Theta$ is convex compact and $(\theta_0, \sigma_0)$ lies in the interior of $\Theta$.
\item[\bf A6.] The function $a$ and all its $x$-derivatives are three times differentiable with respect to $\theta$ for all $x$.  Moreover, these derivatives up to the third order with
respect to $\theta$ belong to $\mathcal{P}$.
\item[\bf A7.] If $\mu_0(a(x,\theta)=a(x,\theta_0))=1$, then $
\theta=\theta_0$.
\item[\bf A8.] $\int \pa_\theta a(x,\theta_0)\,\pa_{\theta^T}a(x,\theta_0)
d\mu_0(x)$ is positive definite, where $\pa_\theta a =\pa a/\pa\theta$.
\end{enumerate}
The following theorem is the main result of this paper. The proposed test has the same limiting null distribution as $T_n$.
\begin{thm}\label{T2}
Assume that {\bf A1}-{\bf A8} hold. For each $\A \geq 0$, if $nh_n^2\rightarrow 0$, then under $H_0$,
\begin{eqnarray*}
T_{j,n}^\A
\stackrel{d}{\rightarrow}\sup_{0\leq t \leq 1} |W^o_t|\quad as\quad n\rightarrow \infty, \quad for\ j=1,2.
\end{eqnarray*}
\end{thm}

\vspace{0.3cm}
\noindent{\bf Remark 1.} For the selection of the tuning constant $M$ in $f_{j,M}$, we rely on the fact that under $H_0$, the distribution of $Z_{\A,i}$ is  approximated to $N(0,1)$ as $h_n$ goes to 0.
Depending on the extent of contamination, one can choose $M$ as the squared number of a proper quantile of the standard normal distribution. For example, if it seems that contamination is low or it is not certain of contamination, one can use the 99.5\% quantile, \ie $M=2.576^2$.
Although it is not easy to assess the degree of contamination, we propose to use $M$ in $[z_{0.025}^2, z_{0.005}^2]=[3.84,6.63]$ based on our simulation results.

\vspace{0.3cm}
\noindent{\bf Remark 2.} The proposed test is not suitable for detecting changes in the drift parameter. The main reason seems to be that  a change in $\theta$ could not make a significant effect on $\hat Z_{\A,i}$ when $h_n$ is small. To see this, note  that
\[\hat{Z}_{\A,i}=\frac{X_{t_i}-X_{t_{i-1}}}{\sqrt{h_n}\hat{\sigma}_n^\A}+a(X_{t_{i-1}}, \hat{\theta}_n^\A)O_p(\sqrt{h_n}).\]
Roughly speaking, whatever the estimated value of $\theta$ is, the influence of $a(X_{t_{i-1}}, \hat{\theta}_n^\A)$ on $\hat Z_{\A,i}$  become reduced when $h_n$ is small,
whereas the estimate of $\sigma$ can make great differences in the residuals. This is why the residual-based CUSUM test is not sensitive to the drift parameter change but is sensitive to the dispersion parameter change.  Such tendency that the residual-based test misses a change of certain parameter has been reported, for example, in Lee (2011) for diffusion process and Song and Kang (2018) for ARMA-GARCH models.
As will be seen in the following section, in the cases that only the drift parameter is changed, all tests considered
produce empirical powers close to significance level. This means that the change of the drift parameter does not affect the performance of the tests.  When the tests reject the null hypothesis, one can therefore conclude that the dispersion parameter has changed.

\vspace{0.3cm}
\noindent{\bf Remark 3.} Any other robust estimators satisfying $\sqrt{nh_n}||\hat{\theta}_n-\theta_0||=O_P(1)$ and $\sqrt{n}|\hat{\sigma}_n-\sigma_0|=O_P(1)$ can be used in the test procedure.

\vspace{0.3cm}
\noindent{\bf Remark 4.} Other types of functions can be employed to trim the squared residuals.  For example, Hampel's function in Andrews et al. (1972), which indeed is an intermediate form between $f_{1,M}$ and $f_{2,M}$, can be used. The performance of the test may be different depending on the trimming functions and the tuning constant $M$.  For the Ornstein-Uhlenbeck process, $f_{2,M}$ with $M=z_{0.005}^2$ showed best performance, see the simulation study below.

\vspace{0.3cm}
\noindent{\bf Remark 5.} One may consider the CUSUM of squares test based on the trimmed residuals, that is,
\begin{eqnarray*}
\tilde T_{j,n}^\A:=\frac{1}{\sqrt{n}\hat{\tau}_{g,j}} \max_{1\leq k\leq n } \Big|\sum_{i=1}^k g_{j,M}^2(\hat{Z}_i^\A) -\frac{k}{n}\sum_{i=1}^ng_{j,M}^2(\hat{Z}_i^\A)\Big|,
\end{eqnarray*}
where $g_{j,M}(x)=\textrm{sign}(x) f_{j,M}(|x|)$ and  $\hat{\tau}_{g,j}$ is the sample variance of $\{ g_{j,M}^2(\hat{Z}_i^\A) \}$. In this case, the tuning constant $M$ is chosen as the just quantile of $N(0,1)$. According to our simulation results (not reported), the performances of $T_{j,n}^\A$ and $\tilde T_{j,n}^\A$  are almost similar.  As mentioned above, we present $T_{j,n}^\A$ since it is easier to handle in deriving its asymptotic distribution.
%\vspace{0.3cm}
%\noindent{\bf \textcolor{red}{Remark 4}.} According to the simulation results, the truncation or the elimination procedure do not seem to result in a size distortion. $T_E^\A$ shows more robust properties and performs better than $T_T^\A$, but $T_T^\A$ does not perform as well as anticipated.

%\vspace{0.3cm}
%\noindent{\bf Remark 5.} Our test can be extended to the following type of diffusion process:
%\begin{eqnarray}\label{DP2}
%dX_t =a(X_t,\theta)dt + \sigma\, b(X_t)dW_t\,.
%\end{eqnarray}
%Using the Lamperti transformation of $X_t$ and It\^{o}'s formula, we obtain
%\begin{eqnarray*}
%d f(X_t) = \Big(\frac{a(X_t,\theta)}{b(X_t)}-\frac{\sigma^2}{2}\,\pa_x
%b(X_t)\Big)dt+\sigma dW_t,
%\end{eqnarray*}
%where $f(x)=\int_c^x b^{-1}(z)dz$, and hence our proposed test can be applied to the type of process (\ref{DP2}).

\section{Simulation study}
In the present simulation, we compare the performances of the naive test $T_n$ and the proposed tests  $T_{1,n}^\A$ and $T_{2,n}^\A$.
For this task, we consider the following Ornstein-Uhlenbeck (OU) process:
\begin{eqnarray}\label{3.1}
dX_t = - \theta X_t dt + {\sigma} dW_t ,
\qquad X_0 = 0.
\end{eqnarray}
 The sample $\{ X_{t_i}\}_{i=0}^n$ is obtained with the sampling interval of $h_n=n^{-0.75}$, where the path of $X$ is generated via the Euler scheme with the generating interval of $h_n/20$. % We note that when the sample size is 1000, $h_n=1000^{-0.75} \approx 1.5/250$ corresponds to the interval of 1.5 trading days in financial applications and thus it can be considered as a high-frequency case.
For the tuning constant $M$, we use $z_{0.005}^2$(=6.63) and $z_{0.025}^2$(=3.84).  %The sample sizes under consideration are $n$ = 500 and 1000.
To evaluate the empirical sizes, we generate paths with $(\theta, \sigma)$=(1,1). For the powers, we change the parameter $(\theta,\sigma)$ from (1,1) to (1,1.2), (1,1.5), (5,1), and (5,1.2) at the midpoint $t=nh_n /2$.  Empirical sizes and powers are calculated at 5\% significance level, based on 5,000 repetitions.
The corresponding critical value is 1.358, which is obtained from the following well-known formula:
\[P\left\{ \sup_{0\leq t \leq 1} |W^o_t| \leq u \right\} =\sum_{k=-\infty}^{\infty} (-1)^k \exp(-2k^2u^2).\]
%Thus, we reject $H_0$ if the test statistics is larger than 1.358.
%The empirical sizes and powers are calculated as the ratio of the rejection number of the $H_0$ out of 10,000 repetitions.

\begin{table}[t]
\caption{\small  Empirical sizes of $T_n$, $T_{1,n}^\A$ and $T_{2,n}^\A$ in the case of $h_n=n^{-0.75}$ and no contamination}\label{Tab1}\vspace{0.1cm}
\tabcolsep=4pt
\renewcommand{\arraystretch}{1.05}
\centering
{\footnotesize
\begin{tabular}{cccccccccccccc}
\toprule
      &       &       & \multicolumn{5}{c}{$M=z_{0.005}^2$}                &       & \multicolumn{5}{c}{$M=z_{0.025}^2$} \\
\cmidrule{4-8}\cmidrule{10-14} $n$     & $T_n$    & $T_{j,n}^\A$    & $\A=0.1$   & $\A=0.2$   & $\A=0.3$  & $\A=0.5$   & $\A=1$     &
& $\A=0.1$   & $\A=0.2$   & $\A=0.3$   & $\A=0.5$   & $\A=1$ \\
\midrule
200   & 0.039 & $T_{1,n}^\A$    & 0.037 & 0.037 & 0.037 & 0.037 & 0.037 &       & 0.042 & 0.041 & 0.041 & 0.039 & 0.039 \\
      &       & $T_{2,n}^\A$    & 0.038 & 0.039 & 0.038 & 0.037 & 0.038 &       & 0.042 & 0.041 & 0.041 & 0.041 & 0.040 \\
\midrule
500   & 0.045 & $T_{1,n}^\A$    & 0.046 & 0.046 & 0.046 & 0.047 & 0.047 &       & 0.047 & 0.047 & 0.048 & 0.048 & 0.048 \\
      &       & $T_{2,n}^\A$    & 0.048 & 0.049 & 0.048 & 0.048 & 0.049 &       & 0.051 & 0.051 & 0.051 & 0.049 & 0.049 \\
\midrule
1000  & 0.046 & $T_{1,n}^\A$    & 0.044 & 0.045 & 0.045 & 0.045 & 0.046 &       & 0.047 & 0.048 & 0.048 & 0.048 & 0.048 \\
      &       & $T_{2,n}^\A$    & 0.045 & 0.044 & 0.044 & 0.045 & 0.045 &       & 0.047 & 0.047 & 0.047 & 0.047 & 0.047 \\
\midrule
3000  & 0.049 & $T_{1,n}^\A$    & 0.047 & 0.047 & 0.047 & 0.048 & 0.048 &       & 0.047 & 0.046 & 0.046 & 0.046 & 0.046 \\
      &       & $T_{2,n}^\A$    & 0.047 & 0.048 & 0.048 & 0.048 & 0.048 &       & 0.048 & 0.048 & 0.048 & 0.048 & 0.047  \\
\bottomrule
\end{tabular}}
\end{table}\vspace{0cm}

\begin{table}[!h]
\caption{\small  Empirical powers of $T_n$, $T_{1,n}^\A$ and $T_{2,n}^\A$ without outliers when $(\theta_0,\sigma_0)$ changes from (1,1) to $(\theta_1,\sigma_1)$}\label{Tab2}\vspace{0.1cm}
\tabcolsep=3pt
\renewcommand{\arraystretch}{1.05}
\centering
{\footnotesize
\begin{tabular}{ccccccccccccccc}
\toprule
                      &          &         &             & \multicolumn{5}{c}{$M=z_{0.005}^2$} &&\multicolumn{5}{c}{$M=z_{0.025}^2$}  \\
\cmidrule{5-9}\cmidrule{11-15}
$(\theta_1,\sigma_1)$ &     $n$    &  $T_n$  & $T_{j,n}^\A$   & $\A=0.1$ & $\A=0.2$ & $\A=0.3$ & $\A=0.5$ & $\A=1.0$&& $\A=0.1$ & $\A=0.2$ & $\A=0.3$ & $\A=0.5$ & $\A=1.0$  \\\midrule
\multirow{6}{*}{(1,1.2)} & 200   & 0.297 &$T_{1,n}^\A$     & 0.303 & 0.302 & 0.302 & 0.301 & 0.301 &       & 0.288 & 0.287 & 0.287 & 0.287 & 0.288 \\
                         &       &       &$T_{2,n}^\A$     & 0.280 & 0.279 & 0.279 & 0.275 & 0.272 &       & 0.191 & 0.189 & 0.186 & 0.184 & 0.180 \\
\cmidrule{2-15}      & 500   & 0.708 & $T_{1,n}^\A$     & 0.699 & 0.700 & 0.699 & 0.699 & 0.698 &       & 0.664 & 0.662 & 0.662 & 0.660 & 0.658 \\
                     &       &       & $T_{2,n}^\A$     & 0.664 & 0.662 & 0.659 & 0.657 & 0.653 &       & 0.451 & 0.448 & 0.444 & 0.441 & 0.433 \\
\cmidrule{2-15}      & 1000  & 0.957 & $T_{1,n}^\A$     & 0.953 & 0.953 & 0.953 & 0.953 & 0.953 &       & 0.936 & 0.936 & 0.935 & 0.934 & 0.933 \\
                     &       &       & $T_{2,n}^\A$     & 0.934 & 0.933 & 0.934 & 0.932 & 0.930 &       & 0.755 & 0.751 & 0.747 & 0.744 & 0.739 \\
\midrule
\multirow{6}{*}{(1,1.5)} & 200   & 0.922 & $T_{1,n}^\A$     & 0.929 & 0.929 & 0.928 & 0.927 & 0.925 &       & 0.914 & 0.912 & 0.912 & 0.911 & 0.909 \\
                         &       &       & $T_{2,n}^\A$     & 0.904 & 0.901 & 0.898 & 0.894 & 0.889 &       & 0.746 & 0.734 & 0.718 & 0.693 & 0.662 \\
\cmidrule{2-15}      & 500   & 1.000 & $T_{1,n}^\A$     & 1.000 & 1.000 & 1.000 & 1.000 & 1.000 &       & 1.000 & 1.000 & 1.000 & 1.000 & 1.000 \\
                     &       &       & $T_{2,n}^\A$     & 1.000 & 1.000 & 1.000 & 1.000 & 1.000 &       & 0.990 & 0.988 & 0.985 & 0.978 & 0.971 \\
\cmidrule{2-15}      & 1000  & 1.000 & $T_{1,n}^\A$     & 1.000 & 1.000 & 1.000 & 1.000 & 1.000 &       & 1.000 & 1.000 & 1.000 & 1.000 & 1.000 \\
                     &       &       & $T_{2,n}^\A$     & 1.000 & 1.000 & 1.000 & 1.000 & 1.000 &       & 1.000 & 1.000 & 1.000 & 1.000 & 1.000 \\
\midrule
\multirow{6}{*}{(5,1)} & 200   & 0.036 & $T_{1,n}^\A$     & 0.039 & 0.039 & 0.039 & 0.040 & 0.041 &       & 0.040 & 0.039 & 0.039 & 0.039 & 0.041 \\
                       &       &       & $T_{2,n}^\A$     & 0.039 & 0.038 & 0.037 & 0.036 & 0.036 &       & 0.039 & 0.040 & 0.039 & 0.038 & 0.037 \\
\cmidrule{2-15}      & 500   & 0.048 & $T_{1,n}^\A$     & 0.049 & 0.049 & 0.049 & 0.049 & 0.049 &       & 0.046 & 0.046 & 0.047 & 0.046 & 0.048 \\
                     &       &       & $T_{2,n}^\A$     & 0.045 & 0.045 & 0.045 & 0.046 & 0.047 &       & 0.041 & 0.042 & 0.042 & 0.042 & 0.042 \\
\cmidrule{2-15}      & 1000  & 0.043 & $T_{1,n}^\A$     & 0.045 & 0.045 & 0.045 & 0.045 & 0.044 &       & 0.046 & 0.046 & 0.045 & 0.044 & 0.044 \\
                     &       &       & $T_{2,n}^\A$     & 0.045 & 0.045 & 0.045 & 0.045 & 0.045 &       & 0.044 & 0.044 & 0.043 & 0.043 & 0.042 \\
\midrule
\multirow{6}{*}{(5,1.2)} & 200   & 0.210 & $T_{1,n}^\A$     & 0.215 & 0.216 & 0.216 & 0.216 & 0.217 &       & 0.203 & 0.204 & 0.205 & 0.206 & 0.207 \\
                         &       &       & $T_{2,n}^\A$     & 0.201 & 0.202 & 0.201 & 0.201 & 0.198 &       & 0.133 & 0.132 & 0.131 & 0.130 & 0.125 \\
\cmidrule{2-15}      & 500   & 0.628 & $T_{1,n}^\A$     & 0.629 & 0.629 & 0.629 & 0.629 & 0.627 &       & 0.588 & 0.589 & 0.589 & 0.589 & 0.588 \\
                     &       &       & $T_{2,n}^\A$     & 0.583 & 0.583 & 0.582 & 0.582 & 0.579 &       & 0.393 & 0.393 & 0.390 & 0.388 & 0.382 \\
\cmidrule{2-15}      & 1000  & 0.940 & $T_{1,n}^\A$     & 0.939 & 0.938 & 0.938 & 0.938 & 0.937 &       & 0.910 & 0.910 & 0.910 & 0.910 & 0.909 \\
                     &       &       & $T_{2,n}^\A$     & 0.908 & 0.907 & 0.907 & 0.906 & 0.907 &       & 0.719 & 0.717 & 0.716 & 0.714 & 0.709 \\
\bottomrule
\end{tabular}}
\end{table}\vspace{0cm}

We first examine the case where the data is not contaminated by outliers. The empirical sizes and powers  are presented in Tables \ref{Tab1} and \ref{Tab2}, respectively.  One can see that $T_n$, $T_{1,n}^\A$, and $T_{2,n}^\A$ show no size distortions and produce reasonably good powers against the change of the dispersion parameter $\sigma$, regardless of whether the drift parameter $\theta$ changes or not. It is, however, observed that all the tests can not detect the change of $\theta$ as mentioned in Remark 2. They produces empirical powers very close to the significance level. $T_n$ and $T_{1,n}^\A$ with $M=z_{0.005}^2$ perform similarly, and $T_{1,n}^\A$ with $M=z_{0.025}^2$ and $T_{2,n}$ with $M=z_{0.005}^2$ are found to yield slightly smaller powers. $T_{2,n}^\A$ with $M=z_{0.025}^2$ is comparatively less powerful, but its power approaches 1 as the sample size increases.  Interestingly,  MDPDE's tuning parameter $\A$ does not make a significant difference in the performance. As will be seen in the contaminated cases below, the performance of the proposed tests are also not significantly different depending on the value of $\A$, so the choice of $\A$ does not seem to be critical in the testing procedure.
Although not reported here, we can see that the powers tend to decrease with a decrease in $M$ but no size distortions are found.

%Table \ref{tb21} presents the powers for the case of $M=z_{0.25}^2$. From the table, we can see  significant power losses of $T_{1,n}^\A$ and $T_{2,n}^\A$. Particularly, the loss of $T_{2,n}^\A$ is found to be greater than $T_{1,n}^\A$.

\begin{table}[t]
\caption{\small  Empirical sizes of $T_n$, $T_{1,n}^\A$ and $T_{2,n}^\A$ under the contamination with  $p=0.5\%$ and $\sigma_v^2=1$}\label{Tab3}\vspace{0.1cm}
\tabcolsep=4pt
\renewcommand{\arraystretch}{1.05}
\centering
{\footnotesize
\begin{tabular}{cccccccccccccc}
\toprule
      &         &             & \multicolumn{5}{c}{$M=z_{0.005}^2$} &&\multicolumn{5}{c}{$M=z_{0.025}^2$}  \\
\cmidrule{4-8}\cmidrule{10-14}
 $n$    &  $T_n$  & $T_{j,n}^\A$   & $\A=0.1$ & $\A=0.2$ & $\A=0.3$ & $\A=0.5$ & $\A=1.0$&& $\A=0.1$ & $\A=0.2$ & $\A=0.3$ & $\A=0.5$ & $\A=1.0$  \\
\midrule
200   & 0.051 & $T_{1,n}^\A$    & 0.054 & 0.051 & 0.051 & 0.051 & 0.052 &       & 0.051 & 0.049 & 0.048 & 0.048 & 0.049 \\
      &       & $T_{2,n}^\A$    & 0.039 & 0.039 & 0.039 & 0.039 & 0.039 &       & 0.041 & 0.040 & 0.041 & 0.041 & 0.040 \\
\midrule
500   & 0.100 & $T_{1,n}^\A$    & 0.061 & 0.061 & 0.061 & 0.061 & 0.062 &       & 0.055 & 0.055 & 0.055 & 0.055 & 0.056 \\
      &       & $T_{2,n}^\A$    & 0.046 & 0.046 & 0.046 & 0.046 & 0.047 &       & 0.047 & 0.047 & 0.047 & 0.049 & 0.049 \\
\midrule
1000  & 0.157 & $T_{1,n}^\A$    & 0.068 & 0.067 & 0.067 & 0.068 & 0.069 &       & 0.053 & 0.053 & 0.053 & 0.053 & 0.053 \\
      &       & $T_{2,n}^\A$    & 0.038 & 0.037 & 0.038 & 0.037 & 0.037 &       & 0.044 & 0.044 & 0.044 & 0.045 & 0.046 \\
%\midrule
%2000  & 0.191 & $T_{1,n}^\A$    & 0.070 & 0.070 & 0.070 & 0.069 & 0.070 &       & 0.058 & 0.057 & 0.057 & 0.057 & 0.057 \\
%      &       & $T_{2,n}^\A$    & 0.044 & 0.044 & 0.045 & 0.045 & 0.044 &       & 0.049 & 0.050 & 0.049 & 0.049 & 0.048 \\
\midrule
3000  & 0.213 & $T_{1,n}^\A$    & 0.078 & 0.078 & 0.078 & 0.077 & 0.078 &       & 0.062 & 0.062 & 0.062 & 0.062 & 0.062 \\
      &       & $T_{2,n}^\A$    & 0.048 & 0.048 & 0.047 & 0.047 & 0.047 &       & 0.049 & 0.048 & 0.048 & 0.048 & 0.048 \\
\bottomrule
   \end{tabular}  }
 \end{table}

 \begin{table}[!h]
\caption{\small  Empirical powers of $T_n$, $T_{1,n}^\A$ and $T_{2,n}^\A$ in the case of $h_n=n^{-0.75}$, $p=0.5\%$ and $\sigma_v^2=1$ when $(\theta_0,\sigma_0)$ changes from (1,1) to $(\theta_1,\sigma_1)$}\label{Tab4}\vspace{0.1cm}
\tabcolsep=3pt
\renewcommand{\arraystretch}{1.05}
\centering
{\footnotesize
% Table generated by Excel2LaTeX from sheet 'Powers'
\begin{tabular}{ccccccccccccccc}
\toprule
                      &          &         &             & \multicolumn{5}{c}{$M=z_{0.005}^2$} &&\multicolumn{5}{c}{$M=z_{0.025}^2$}\\
\cmidrule{5-9}\cmidrule{11-15}
 $(\theta_1,\sigma_1)$&     $n$    &  $T_n$  & $T_{j,n}^\A$   & $\A=0.1$ & $\A=0.2$ & $\A=0.3$ & $\A=0.5$ & $\A=1.0$&& $\A=0.1$ & $\A=0.2$ & $\A=0.3$ & $\A=0.5$ & $\A=1.0$  \\
\midrule
\multirow{6}{*}{(1,1.2)} & 200   & 0.201 & $T_{1,n}^\A$     & 0.294 & 0.294 & 0.292 & 0.289 & 0.290 &       & 0.285 & 0.282 & 0.281 & 0.280 & 0.277 \\
      &       &       & $T_{2,n}^\A$    & 0.274 & 0.272 & 0.269 & 0.267 & 0.262 &       & 0.191 & 0.186 & 0.182 & 0.177 & 0.173 \\
\cmidrule{2-15}      & 500   & 0.274 & $T_{1,n}^\A$    & 0.634 & 0.635 & 0.637 & 0.638 & 0.639 &       & 0.631 & 0.629 & 0.628 & 0.628 & 0.625 \\
      &       &       & $T_{2,n}^\A$    & 0.645 & 0.640 & 0.637 & 0.634 & 0.632 &       & 0.456 & 0.442 & 0.438 & 0.431 & 0.425 \\
\cmidrule{2-15}      & 1000  & 0.224 & $T_{1,n}^\A$    & 0.909 & 0.909 & 0.909 & 0.910 & 0.910 &       & 0.914 & 0.914 & 0.913 & 0.914 & 0.913 \\
      &       &       & $T_{2,n}^\A$    & 0.927 & 0.926 & 0.925 & 0.924 & 0.924 &       & 0.770 & 0.757 & 0.752 & 0.745 & 0.740 \\
\midrule
\multirow{6}{*}{(1,1.5)} & 200   & 0.646 & $T_{1,n}^\A$    & 0.873 & 0.878 & 0.880 & 0.880 & 0.881 &       & 0.887 & 0.885 & 0.885 & 0.883 & 0.880 \\
      &       &       & $T_{2,n}^\A$    & 0.896 & 0.892 & 0.888 & 0.882 & 0.875 &       & 0.739 & 0.718 & 0.700 & 0.677 & 0.644 \\
\cmidrule{2-15}      & 500   & 0.507 & $T_{1,n}^\A$    & 0.999 & 0.999 & 0.999 & 0.999 & 0.999 &       & 1.000 & 1.000 & 1.000 & 1.000 & 1.000 \\
      &       &       & $T_{2,n}^\A$    & 1.000 & 1.000 & 1.000 & 1.000 & 1.000 &       & 0.990 & 0.987 & 0.984 & 0.981 & 0.974 \\
\cmidrule{2-15}      & 1000  & 0.377 & $T_{1,n}^\A$    & 1.000 & 1.000 & 1.000 & 1.000 & 1.000 &       & 1.000 & 1.000 & 1.000 & 1.000 & 1.000 \\
      &       &       & $T_{2,n}^\A$    & 1.000 & 1.000 & 1.000 & 1.000 & 1.000 &       & 1.000 & 1.000 & 1.000 & 1.000 & 1.000 \\
\midrule
\multirow{6}{*}{(5,1)} & 200   & 0.050 & $T_{1,n}^\A$    & 0.052 & 0.051 & 0.051 & 0.052 & 0.054 &       & 0.045 & 0.045 & 0.045 & 0.044 & 0.045 \\
      &       &       & $T_{2,n}^\A$    & 0.038 & 0.037 & 0.037 & 0.036 & 0.039 &       & 0.041 & 0.040 & 0.040 & 0.039 & 0.037 \\
\cmidrule{2-15}      & 500   & 0.092 & $T_{1,n}^\A$    & 0.067 & 0.066 & 0.066 & 0.066 & 0.067 &       & 0.058 & 0.057 & 0.057 & 0.057 & 0.056 \\
      &       &       & $T_{2,n}^\A$    & 0.050 & 0.050 & 0.050 & 0.050 & 0.051 &       & 0.044 & 0.044 & 0.043 & 0.043 & 0.043 \\
\cmidrule{2-15}      & 1000  & 0.150 & $T_{1,n}^\A$    & 0.074 & 0.073 & 0.073 & 0.073 & 0.073 &       & 0.061 & 0.060 & 0.061 & 0.061 & 0.061 \\
      &       &       & $T_{2,n}^\A$    & 0.048 & 0.048 & 0.048 & 0.048 & 0.048 &       & 0.048 & 0.047 & 0.047 & 0.047 & 0.047 \\
\midrule
\multirow{6}{*}{(5,1.2)} & 200   & 0.158 & $T_{1,n}^\A$    & 0.225 & 0.225 & 0.225 & 0.226 & 0.226 &       & 0.211 & 0.212 & 0.214 & 0.213 & 0.211 \\
      &       &       & $T_{2,n}^\A$    & 0.200 & 0.202 & 0.201 & 0.201 & 0.201 &       & 0.136 & 0.135 & 0.134 & 0.132 & 0.131 \\
\cmidrule{2-15}      & 500   & 0.256 & $T_{1,n}^\A$    & 0.561 & 0.567 & 0.567 & 0.569 & 0.570 &       & 0.566 & 0.566 & 0.565 & 0.564 & 0.563 \\
      &       &       & $T_{2,n}^\A$    & 0.570 & 0.569 & 0.569 & 0.569 & 0.567 &       & 0.400 & 0.391 & 0.389 & 0.389 & 0.388 \\
\cmidrule{2-15}      & 1000  & 0.219 & $T_{1,n}^\A$    & 0.877 & 0.879 & 0.880 & 0.881 & 0.880 &       & 0.882 & 0.884 & 0.884 & 0.885 & 0.883 \\
      &       &       & $T_{2,n}^\A$    & 0.899 & 0.897 & 0.896 & 0.895 & 0.895 &       & 0.723 & 0.716 & 0.712 & 0.709 & 0.706 \\
\bottomrule
\end{tabular}}
 \end{table}\vspace{0cm}

\begin{table}[h]
\caption{\small  Empirical sizes of $T_n$, $T_{1,n}^\A$ and $T_{2,n}^\A$ in the case of $h_n=n^{-0.75}$, $p=5\%$ and $\sigma_v^2=1$}\label{Tab5}\vspace{0.1cm}
\tabcolsep=4pt
\renewcommand{\arraystretch}{1.05}
\centering
{\footnotesize
\begin{tabular}{cccccccccccccc}
\toprule
      &         &             & \multicolumn{5}{c}{$M=z_{0.005}^2$} &&\multicolumn{5}{c}{$M=z_{0.025}^2$}  \\
\cmidrule{4-8}\cmidrule{10-14}
 $n$    &  $T_n$  & $T_{j,n}^\A$   & $\A=0.1$ & $\A=0.2$ & $\A=0.3$ & $\A=0.5$ & $\A=1.0$&& $\A=0.1$ & $\A=0.2$ & $\A=0.3$ & $\A=0.5$ & $\A=1.0$  \\
\midrule
200   & 0.137 & $T_{1,n}^\A$    & 0.167 & 0.148 & 0.140 & 0.138 & 0.139 &       & 0.140 & 0.111 & 0.106 & 0.105 & 0.106 \\
      &       & $T_{2,n}^\A$    & 0.075 & 0.055 & 0.053 & 0.053 & 0.053 &       & 0.058 & 0.048 & 0.047 & 0.047 & 0.047 \\
\midrule
500   & 0.196 & $T_{1,n}^\A$    & 0.198 & 0.170 & 0.165 & 0.165 & 0.166 &       & 0.149 & 0.123 & 0.116 & 0.117 & 0.120 \\
      &       & $T_{2,n}^\A$    & 0.065 & 0.054 & 0.054 & 0.053 & 0.053 &       & 0.059 & 0.050 & 0.050 & 0.049 & 0.049 \\
\midrule
1000  & 0.217 & $T_{1,n}^\A$    & 0.203 & 0.181 & 0.178 & 0.176 & 0.181 &       & 0.152 & 0.127 & 0.125 & 0.124 & 0.128 \\
      &       & $T_{2,n}^\A$    & 0.068 & 0.058 & 0.057 & 0.056 & 0.057 &       & 0.054 & 0.052 & 0.053 & 0.053 & 0.052 \\
%\midrule
%2000  & 0.242 & $T_{1,n}^\A$    & 0.214 & 0.197 & 0.194 & 0.195 & 0.200 &       & 0.158 & 0.144 & 0.142 & 0.141 & 0.145 \\
%      &       & $T_{2,n}^\A$    & 0.059 & 0.057 & 0.056 & 0.056 & 0.058 &       & 0.058 & 0.056 & 0.057 & 0.057 & 0.058 \\
\midrule
3000  & 0.254 & $T_{1,n}^\A$    & 0.210 & 0.196 & 0.194 & 0.195 & 0.202 &       & 0.145 & 0.132 & 0.129 & 0.130  & 0.135 \\
      &       & $T_{2,n}^\A$    & 0.064 & 0.062 & 0.061 & 0.062 & 0.061 &       & 0.058 & 0.059 & 0.060 & 0.060  & 0.058 \\
\bottomrule
   \end{tabular}  }
 \end{table}

 \begin{table}[!h]
\caption{\small  Empirical powers of $T_n$, $T_{1,n}^\A$ and $T_{2,n}^\A$ in the case of $h_n=n^{-0.75}$, $p=5\%$ and $\sigma_v^2=1$ when $(\theta_0,\sigma_0)$ changes from (1,1) to $(\theta_1,\sigma_1)$}\label{Tab6}\vspace{0.1cm}
\tabcolsep=3pt
\renewcommand{\arraystretch}{1.05}
\centering
{\footnotesize
% Table generated by Excel2LaTeX from sheet 'Powers'
\begin{tabular}{ccccccccccccccc}
\toprule
                      &          &         &             & \multicolumn{5}{c}{$M=z_{0.005}^2$} &&\multicolumn{5}{c}{$M=z_{0.025}^2$}\\
\cmidrule{5-9}\cmidrule{11-15}
 $(\theta_1,\sigma_1)$&     n    &  $T_n$  & $T_{j,n}$   & $\A=0.1$ & $\A=0.2$ & $\A=0.3$ & $\A=0.5$ & $\A=1.0$&& $\A=0.1$ & $\A=0.2$ & $\A=0.3$ & $\A=0.5$ & $\A=1.0$  \\
\midrule
\multirow{6}{*}{(1,1.2)} & 200   & 0.151 & $T_{1,n}^\A$    & 0.233 & 0.240 & 0.242 & 0.246 & 0.246 &       & 0.251 & 0.254 & 0.252 & 0.252 & 0.250 \\
      &       &       & $T_{2,n}^\A$    & 0.211 & 0.222 & 0.225 & 0.224 & 0.225 &       & 0.222 & 0.199 & 0.187 & 0.179 & 0.177 \\
\cmidrule{2-15}      & 500   & 0.197 & $T_{1,n}^\A$    & 0.353 & 0.394 & 0.400 & 0.404 & 0.401 &       & 0.425 & 0.472 & 0.477 & 0.478 & 0.476 \\
      &       &       & $T_{2,n}^\A$    & 0.484 & 0.535 & 0.540 & 0.538 & 0.540 &       & 0.512 & 0.452 & 0.429 & 0.419 & 0.425 \\
\cmidrule{2-15}      & 1000  & 0.231 & $T_{1,n}^\A$    & 0.491 & 0.570 & 0.582 & 0.589 & 0.582 &       & 0.641 & 0.698 & 0.705 & 0.707 & 0.703 \\
      &       &       & $T_{2,n}^\A$    & 0.815 & 0.855 & 0.856 & 0.857 & 0.859 &       & 0.825 & 0.761 & 0.736 & 0.727 & 0.740 \\
\midrule
\multirow{6}{*}{(1,1.5)} & 200   & 0.190 & $T_{1,n}^\A$    & 0.444 & 0.518 & 0.548 & 0.568 & 0.574 &       & 0.574 & 0.647 & 0.668 & 0.681 & 0.685 \\
      &       &       & $T_{2,n}^\A$    & 0.612 & 0.717 & 0.744 & 0.754 & 0.752 &       & 0.721 & 0.708 & 0.682 & 0.657 & 0.634 \\
\cmidrule{2-15}      & 500   & 0.225 & $T_{1,n}^\A$    & 0.671 & 0.822 & 0.849 & 0.866 & 0.867 &       & 0.861 & 0.945 & 0.953 & 0.958 & 0.958 \\
      &       &       & $T_{2,n}^\A$    & 0.947 & 0.994 & 0.996 & 0.996 & 0.996 &       & 0.993 & 0.990 & 0.982 & 0.977 & 0.975 \\
\cmidrule{2-15}      & 1000  & 0.244 & $T_{1,n}^\A$    & 0.887 & 0.969 & 0.978 & 0.981 & 0.980 &       & 0.985 & 0.999 & 0.999 & 0.999 & 0.999 \\
      &       &       & $T_{2,n}^\A$    & 0.999 & 1.000 & 1.000 & 1.000 & 1.000 &       & 1.000 & 1.000 & 1.000 & 1.000 & 1.000 \\
\midrule
\multirow{6}{*}{(5,1)} & 200   & 0.123 & $T_{1,n}^\A$    & 0.175 & 0.152 & 0.146 & 0.144 & 0.144 &       & 0.148 & 0.116 & 0.110 & 0.107 & 0.110 \\
      &       &       & $T_{2,n}^\A$    & 0.082 & 0.060 & 0.056 & 0.053 & 0.054 &       & 0.077 & 0.047 & 0.044 & 0.042 & 0.042 \\
\cmidrule{2-15}      & 500   & 0.171 & $T_{1,n}^\A$    & 0.206 & 0.177 & 0.173 & 0.170 & 0.174 &       & 0.156 & 0.128 & 0.120 & 0.120 & 0.122 \\
      &       &       & $T_{2,n}^\A$    & 0.077 & 0.061 & 0.060 & 0.059 & 0.060 &       & 0.065 & 0.059 & 0.060 & 0.061 & 0.060 \\
\cmidrule{2-15}      & 1000  & 0.206 & $T_{1,n}^\A$    & 0.208 & 0.187 & 0.182 & 0.182 & 0.187 &       & 0.158 & 0.135 & 0.131 & 0.131 & 0.134 \\
      &       &       & $T_{2,n}^\A$    & 0.072 & 0.060 & 0.058 & 0.058 & 0.058 &       & 0.057 & 0.055 & 0.056 & 0.056 & 0.056 \\
\midrule
\multirow{6}{*}{(5,1.2)} & 200   & 0.123 & $T_{1,n}^\A$    & 0.208 & 0.216 & 0.221 & 0.223 & 0.224 &       & 0.206 & 0.216 & 0.217 & 0.217 & 0.221 \\
      &       &       & $T_{2,n}^\A$    & 0.135 & 0.152 & 0.163 & 0.165 & 0.169 &       & 0.125 & 0.136 & 0.134 & 0.134 & 0.134 \\
\cmidrule{2-15}      & 500   & 0.186 & $T_{1,n}^\A$    & 0.314 & 0.355 & 0.363 & 0.367 & 0.366 &       & 0.353 & 0.406 & 0.415 & 0.420 & 0.418 \\
      &       &       & $T_{2,n}^\A$    & 0.378 & 0.460 & 0.465 & 0.467 & 0.470 &       & 0.400 & 0.375 & 0.365 & 0.361 & 0.369 \\
\cmidrule{2-15}      & 1000  & 0.201 & $T_{1,n}^\A$    & 0.474 & 0.552 & 0.567 & 0.571 & 0.564 &       & 0.603 & 0.668 & 0.681 & 0.685 & 0.680 \\
      &       &       & $T_{2,n}^\A$    & 0.753 & 0.804 & 0.807 & 0.809 & 0.811 &       & 0.757 & 0.705 & 0.684 & 0.676 & 0.693 \\
\bottomrule
\end{tabular}  }
 \end{table}

Next, to explore the cases where outliers are involved in the data, we generate contaminated sample $\{X_{t_i}^c\}_{i=0}^n$
 by the following scheme: $X_{t_i}^c\,=X_{t_i}\,+\,p_i\,|V_i| \times sign(X_{t_i})\,$, where $\{p_i\}_{i=0}^n$ and $\{V_i\}_{i=0}^n$ are sequences of i.i.d. random variables from  Bernoulli distribution with success probability $p$ and normal distribution with mean 0 and variance $\sigma_v^2$, respectively; $\{X_{t_i}\}$, $\{p_i\}$, and $\{V_i\}$ are assumed to be all independent.
 We consider the cases of $p=0.5\%$ and $5\%$ to describe a low degree of contamination and more severely contaminated situation, respectively, and  $\sigma_v^2$ is set to 1. %Table {\ref{tb3}} - {\ref{tb6}} present the results for the cases of $p=0.005, 0.03$ and $\sigma_v^2=2$. $p=0.005$ would describe the situation in that the degree of the contamination is low and $p=0.03$ is selected to present the more contaminated situations.
 The empirical sizes and powers are provided in Tables \ref{Tab3} - \ref{Tab6}. We first note that $T_n$ exhibits size distortions and significant power losses whereas $T_{2,n}$ dramatically eliminates the impact of outliers. $T_{2,n}$ with $M=z_{0.005}^2$ outperforms other tests in all cases considered. On the other hand, $T_{1,n}$ shows relatively good performance in the low contaminated case, i.e., $p=0.5\%$, but are observed to be somewhat affected by outliers when $p=5\%$. In this case, $T_{1,n}$ with the smaller $M$(= $z_{0.025}^2$)  shows more robust behavior  than $T_{1,n}$ with $z_{0.005}^2$.

\begin{figure}[!t]
\includegraphics[height=0.5\textwidth,width=1.0\textwidth]{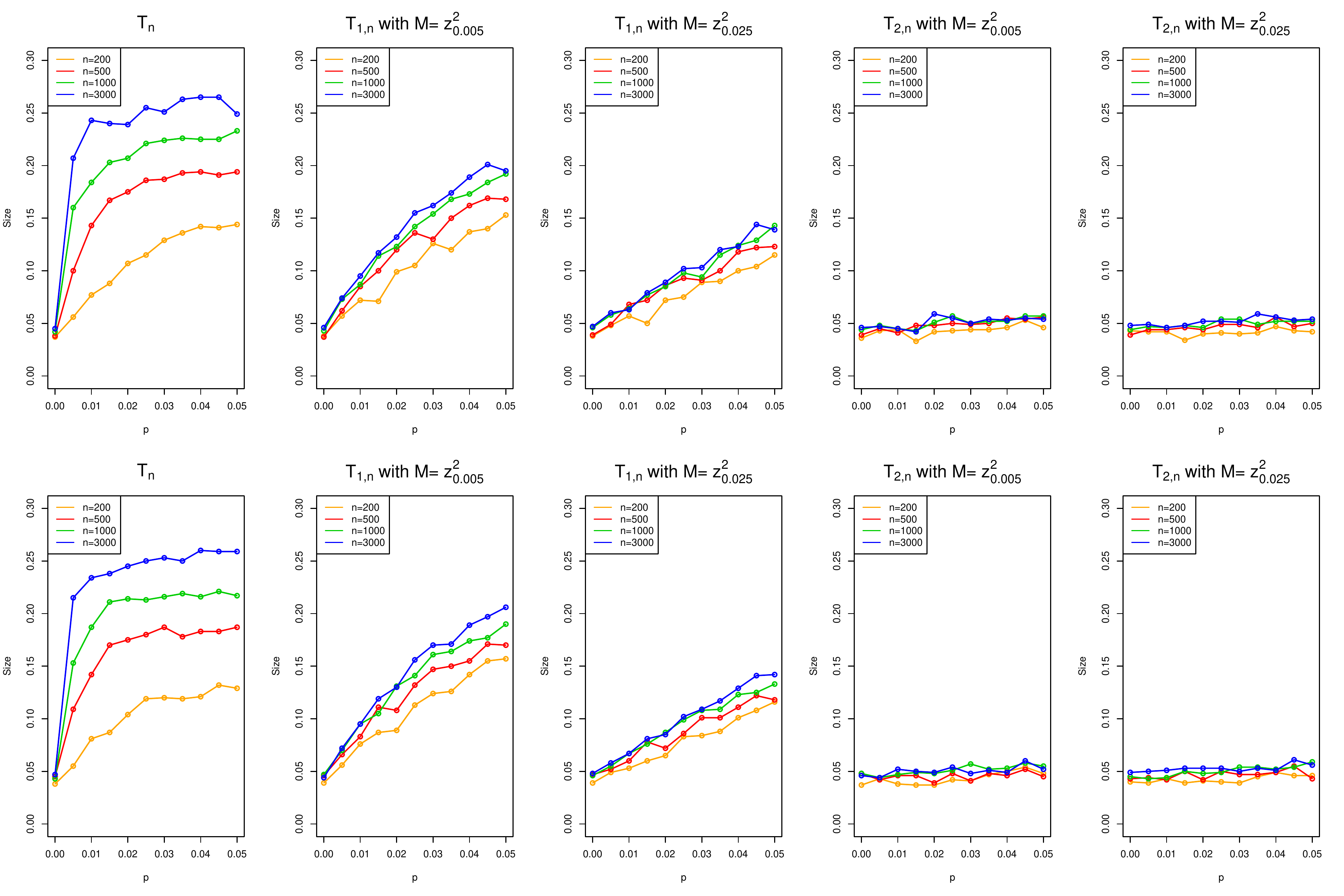}\vspace{-0.1cm}
\caption{\small The plots of the empirical sizes of the naive test, $T_n$,  and the proposed test, $T_{j,n}^{\A=0.2}$.  The upper panel presents  the sizes in the case of $\sigma_v^2=1$ and the lower panel for $\sigma^2_v=2$. }\label{fig1}
\vspace{1.2cm}
\includegraphics[height=0.5\textwidth,width=1\textwidth]{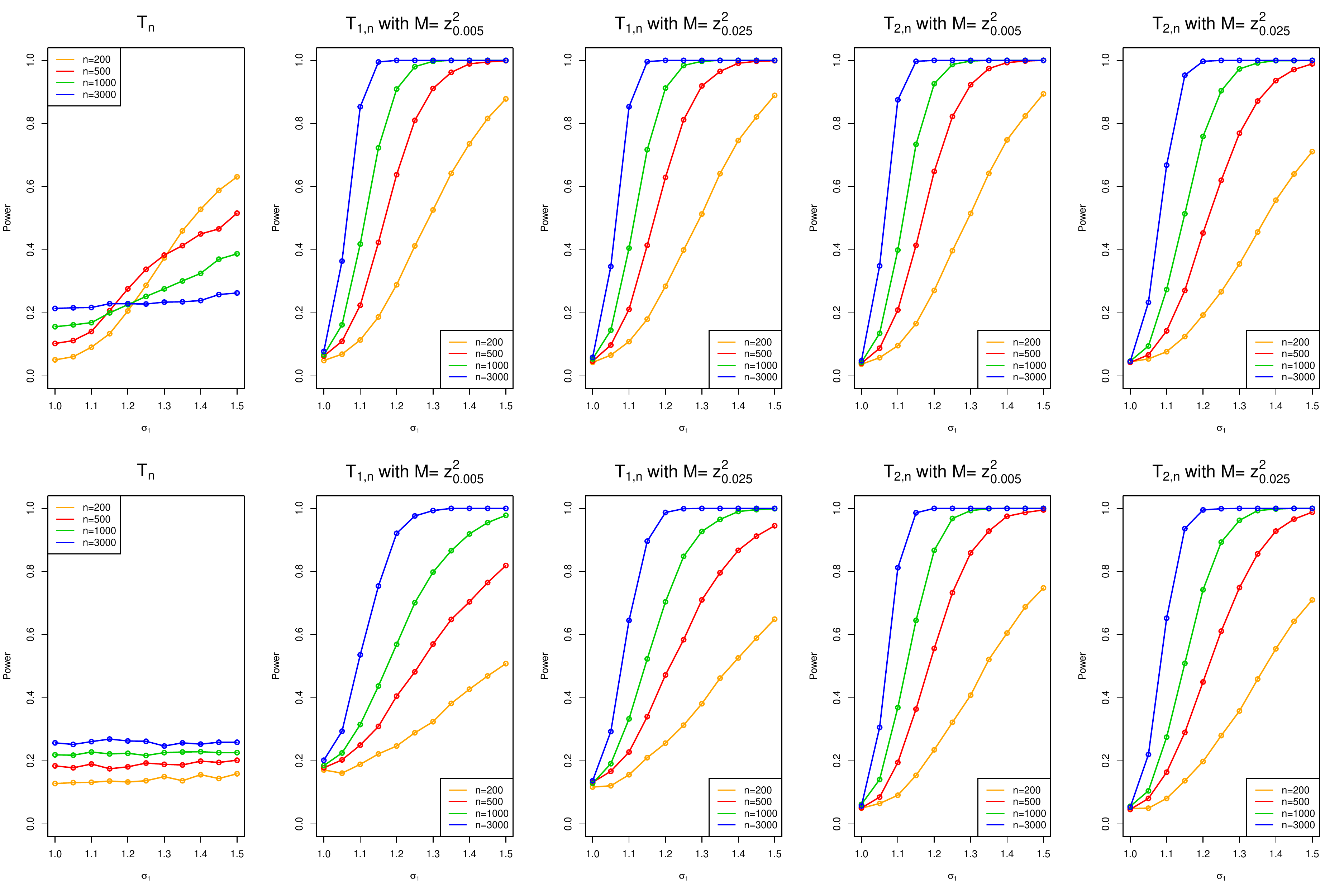} \vspace{-0.4cm}
\caption{\small The plots of the empirical powers of the naive test, $T_n$,  and the proposed test, $T_{j,n}^{\A=0.2}$.  The  upper panel presents the plots for the case of $p=0.5\%, \sigma_v^2=1$ (low contaminated case) and the lower panel  for $p=5\%, \sigma_v^2=2$ (severely contaminated case). }\label{fig2}
\end{figure}

Our findings  can evidently be  seen in  Figures \ref{fig1} and \ref{fig2}. Figure \ref{fig1} presents  the plots of the empirical sizes of $T_n$ and $T_{j,n}^{\A=0.2}$ versus the degree of contamination $p \in \{0\%, 0.05\%, 1\%,\cdots, 5\%\}$. The upper and lower panels depict the results in the cases of $\sigma_v^2=1$ and  $\sigma_v^2=2$, respectively.  One can clearly see the severe size distortions of $T_n$ in the first column in Figure \ref{fig1}.  As can be seen in Table \ref{Tab3}, $T_n$ is damaged even  by the small portion of contamination, i.e., $p=0.5\%$.  It should also be noted that the distortion gets worse as $n$ increases, indicating that particular attention should be paid when dealing with high-frequency data. This is because
the closer the transition distribution gets to normal distribution, the more affected it is by outliers. Upward trends in the second and third columns suggest  that  $T_{1,n}^\A$ may not be suitable for high-contaminated cases.  The last two columns show strong robustness of $T_{2,n}^\A$ yielding no size distortions in all cases.  Figure \ref{fig2} displays the power curves when $\sigma$ changes from 1 to $\sigma_1$ at midpoint. The upper and lower panels presents for the case of $p=0.5\%, \sigma_v^2=1$ (low contaminated case) and $p=5\%, \sigma_v^2=2$ (severely contaminated case), respectively.  Unlike $T_n$ showing  power losses,   $T_{1,n}^\A$ and $T_{1,n}^\A$  yield reasonable powers in both cases.  In particular, $T_{2,n}^\A$ with $M=z^2_{0.005}$ is observed to perform best.

Overall, the results above  support the validity of the proposed test.
In this simulation  section, we see that our proposed test keeps good sizes and powers in the presence of outliers, while the naive test $T_n$ shows severe size distortions and significant power loses. Therefore, our test can be a functional tool to test for parameter change when outliers are speculated to contaminate data.

\section{Real data analysis}
\vspace{-1cm}
\begin{figure}[!h]
\includegraphics[height=0.5\textwidth,width=1\textwidth]{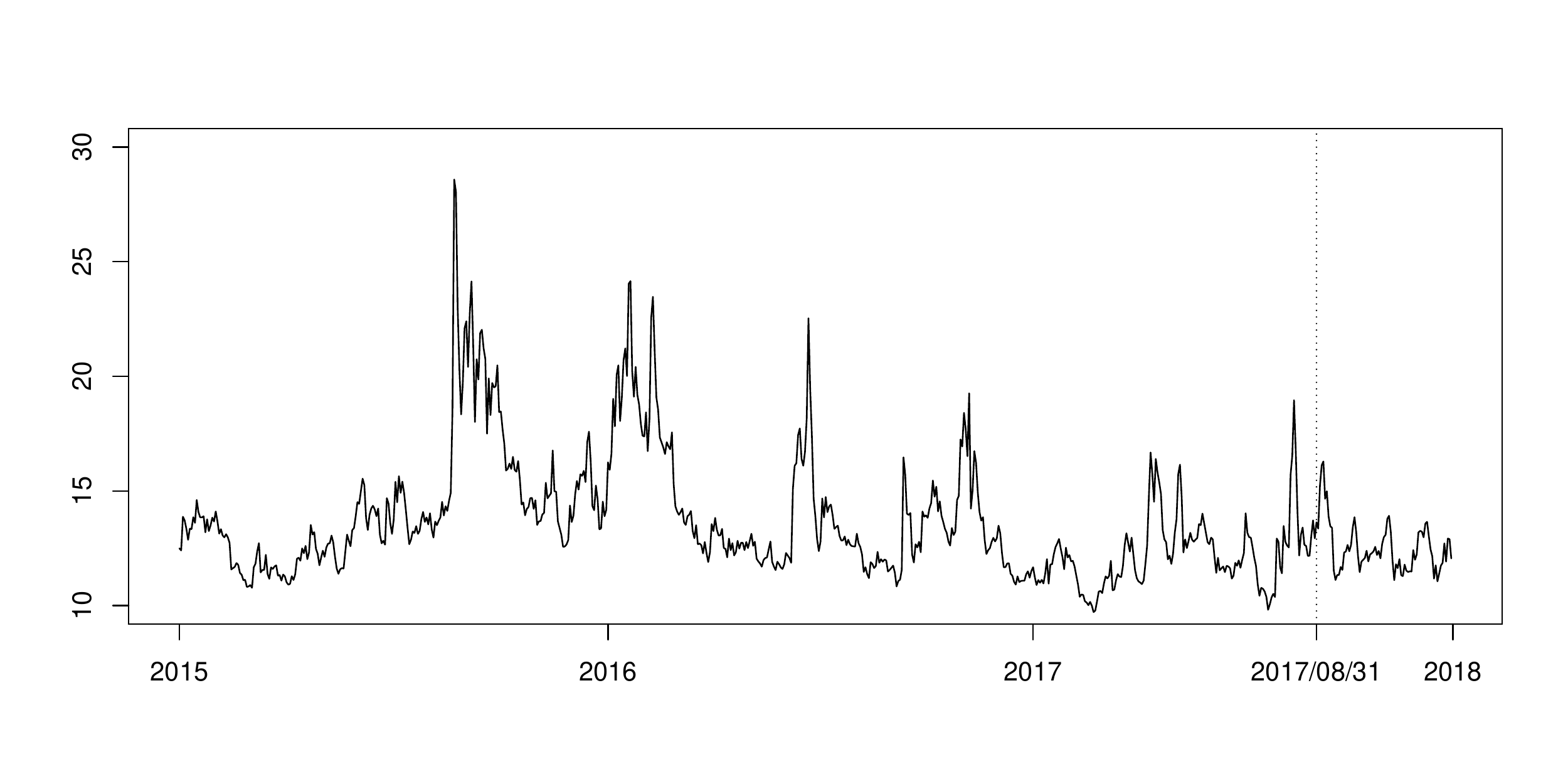}\vspace{-1cm}
\caption{\small Time series  plot of  VKOSPI200 index from Jan 2, 2015 to Dec 28, 2017}\label{VKOSPI}
\end{figure}

We analyse  daily time series of KOSPI200 volatility (VKSOPI200) index. Like the VIX index (the Chicago Board Options Exchange volatility index), VKOPSI200 index is designed to measure 30-day expected volatility of  KOSPI200 index. %The index is a kind of implied volatility obtained  using near ATM  call and put options.
The data analyzed here is depicted in Figure \ref{VKOSPI}, covering 737 trading days from Jan 2, 2015 to Dec  28, 2017.  As a key characteristic, the plot clearly shows a volatile and mean-reverting behaviour. One can also see  a number of  deviating observations. To capture the mean-reversion, we employ the following OU process:
\[ d X_t = \lambda(\mu-  X_t) dt + {\sigma} dW_t,\]
where $X_t$ is the value of VKOSPI200 index at time $t$. One may consider the OU process with jump component to accommodate the spiky observations,  but in the present analysis, we regard these observations as outliers and examine whether or not there were parameter changes in the fitted model.
We conduct the naive test $T_n$ and the proposed test $T_{1,n}^\A$ and $T_{2,n}^\A$ to the data until Aug 31, 2017. Relying on the results in the simulation study, our decision is, however, made based on $T_{n,2}^\A$ with $M=z_{0.005}^2$. The remaining data set after Sep 1, 2017  is used to compare the performances of the models without break and with breaks obtained by $T_{2,n}^\A$.
\begin{table}[!h]
  \centering
  \caption{\small Test statistics [p-values] of $T_n$, $T_{1,n}^\A$, and $T_{2,n}^\A$}\vspace{0.1cm}
  {\footnotesize
    \begin{tabular}{cccccccc}
    \toprule
    $T_n$ & $T_{j,n}^\A$ & $\A=0.1$ & $\A=0.2$ & $\A=0.3$ & $\A=0.4$ & $\A=0.5$ & $\A=1$ \\
    \midrule
    0.957 & \multicolumn{1}{l}{$T_{1,n}^\A$} & 1.863 & 1.843 & 1.807 & 1.769 & 1.789 & 1.835 \\
    $[0.319]$ &       & [0.002] & [0.002] & [0.003] & [0.004] & [0.003] & [0.002] \\
\cmidrule{2-8}          & $T_{2,n}^\A$ & 1.460 & 1.369 & 1.625 & 1.648 & 1.613 & 1.566 \\
          &       & [0.028] & [0.047] & [0.010] & [0.009] & [0.011] & [0.015] \\
    \bottomrule
    \end{tabular}}
  \label{test_results}%
\end{table}

Table \ref{test_results} presents the test statistics and p-values of $T_n$, $T_{1,n}^\A$ and $T_{2,n}^\A$. We first note that the p-value of $T_n$ is obtained to be 0.319 whereas all the p-values of $T_{1,n}^\A$ and $T_{2,n}^\A$ are less then 0.05. As observed in the simulation study, this indicates that the outlying observations are highly likely to have hindered $T_n$ from detecting a significant parameter change. To find further changes, we use the binary segmentation procedure (cf. Aue and Horváth (2013)) and locate three more breaks. The sub-periods divided by the estimated change-points and the ML and MDPD estimates for each sub-period are reported in Table \ref{sub_estimates}.  As mentioned in  Remark 2, the tests are difficult to detect changes in drift parameter and therefore it should be interpreted  that the obtained sub-periods are due to the changes in the dispersion parameter. In Table \ref{sub_estimates}, one can see evident changes in $\hat\sigma$. Although the period is divided by the changes in $\sigma$,  other parameters $\lambda$ and $\mu$ are also estimated differently in each sub-period. % \tcb{Since volatility is sensitive to events and reflect the situation, it is reasonable to detect changes based on dispersion parameter.}
 Here, it is noteworthy that the difference between the ML and the MDPD estimates of $\sigma$ is comparatively large in  the first and the last sub-periods. Since MLE and MDPDE tend to yield similar estimates when the portion of outliers is small, we can surmise that the first and the last periods include some outliers that may affect the ML estimates. %In other periods, there seems to  be no  such observations.
The estimated change-points and $\hat\mu$ are displayed in Figure \ref{chgpts}, where the red and blue dashed lines stand for $\hat\mu$ by MLE and MDPDE with $\A=0.1$, respectively.
% Since, for example, changes of moneytary polcy and critical social events affect financial data like the way changing volatility,  it is meaningful that        . So, our test ... is meaningful.

Finally, we  compare the OU processes without and with parameter changes in terms of the superiority in forecasting performance. Hereafter, we denote the processes with and without changes by \OUchg$ $  and \OU, respectively. Forecasting using the process with changes means that predicted values are obtained using the data from the last change point, i.e., Mar 3, 2016 in the present analysis. We use the Euler approximation in (\ref{Euler}) to calculate one-step-ahead forecasts for the last four months, total 78 observations, as follows:
\begin{eqnarray}\label{predict}
\hat{X}_{t_{s+1}} = X_{t_s}+\hat\lambda_{t_s}(X_{t_s} - \hat\mu_{t_s} )h_n,
\end{eqnarray}
where ($\hat\lambda_{t_s}, \hat\mu_{t_s}$) is an estimate based on the data up to $X_{t_s}$. 95\% prediction interval (PI) is given by $\hat{X}_{t_{s+1}}\pm 2 \hat\sigma_{t_s} \sqrt{h_n}$. The following root mean squared error (RMSE) and the root mean squared percentage error (RMSPE) are employed to evaluate forecasting performance:
\[ \text{RMSE}=\sqrt{\frac{1}{78}\sum_{s=660}^{737}  (X_{t_{s}}- \hat{X}_{t_{s}})^2}\qquad \text{and} \qquad
\text{RMSPE}=\sqrt{\frac{1}{78}\sum_{s=660}^{737} \bigg (\frac{X_{t_{s}}- \hat{X}_{t_{s}}}{X_{t_{s}}}\bigg)^2},\]
where $X_{t_{660}}$ is the index at Sep 1, 2017.

\begin{table}[!h]
  \centering
  \caption{\small ML and MDPD estimates for each sub-period}\vspace{0.1cm}
  \tabcolsep=6pt
\renewcommand{\arraystretch}{0.9}
   {\footnotesize
    \begin{tabular}{ccccrrr}
    \toprule
    \multicolumn{3}{c}{Period} & MDPDE & \multicolumn{1}{c}{$\hat\lambda$} & \multicolumn{1}{c}{$\hat\mu$} & \multicolumn{1}{c}{$\hat\sigma$} \\
    \midrule
    Jan 2, 2015 & $\sim$ &  Aug 25, 2015 & MLE  & 14.93 & 14.23 & 16.57 \\
          &       &       & $\A=0.1$ & 14.23 & 13.05 & 7.88 \\
          &       &       & $\A=0.2$ & 13.99 & 12.87 & 7.53 \\
          &       &       & $\A=0.3$ & 13.53 & 12.76 & 7.37 \\
          &       &       & $\A=0.4$ & 13.11 & 12.68 & 7.28 \\
          &       &       & $\A=0.5$ & 12.72 & 12.62 & 7.23 \\
          &       &       & $\A=1.0$ & 11.38 & 12.33 & 7.24 \\
    \midrule
    Aug 26, 2015 & $\sim$ & Sep 22, 2015 & MLE  & 177.02 & 20.52 & 25.80 \\
          &       &       & $\A=0.1$ & 180.32 & 20.55 & 26.23 \\
          &       &       & $\A=0.2$ & 183.21 & 20.58 & 26.58 \\
          &       &       & $\A=0.3$ & 185.77 & 20.61 & 26.84 \\
          &       &       & $\A=0.4$ & 188.06 & 20.63 & 27.03 \\
          &       &       & $\A=0.5$ & 190.14 & 20.66 & 27.13 \\
          &       &       & $\A=1.0$ & 198.43 & 20.78 & 26.62 \\
    \midrule
    Sep 23, 2015 & $\sim$ & Dec 10, 2015 & MLE  & 31.76 & 14.71 & 10.73 \\
          &       &       & $\A=0.1$ & 30.56 & 14.63 & 10.60 \\
          &       &       & $\A=0.2$ & 28.97 & 14.56 & 10.41 \\
          &       &       & $\A=0.3$ & 26.83 & 14.49 & 10.18 \\
          &       &       & $\A=0.4$ & 23.99 & 14.44 & 9.88 \\
          &       &       & $\A=0.5$ & 20.61 & 14.40 & 9.51 \\
          &       &       & $\A=1.0$ & 11.66 & 14.86 & 8.09 \\
    \midrule
    Dec 11, 2015 & $\sim$ & Mar 2, 2016 & MLE  & 43.43 & 17.88 & 23.28 \\
          &       &       & $\A=0.1$ & 44.94 & 17.55 & 22.56 \\
          &       &       & $\A=0.2$ & 46.46 & 17.28 & 21.83 \\
          &       &       & $\A=0.3$ & 47.81 & 17.08 & 21.24 \\
          &       &       & $\A=0.4$ & 48.97 & 16.93 & 20.88 \\
          &       &       & $\A=0.5$ & 50.06 & 16.83 & 20.70 \\
          &       &       & $\A=1.0$ & 60.70 & 16.51 & 20.20 \\
    \midrule
    Mar 3, 2016 & $\sim$ & Aug 31, 2017 & MLE  & 30.34 & 12.78 & 13.20 \\
          &       &       & $\A=0.1$ & 32.43 & 12.32 & 10.31 \\
          &       &       & $\A=0.2$ & 33.74 & 12.05 & 8.55 \\
          &       &       & $\A=0.3$ & 33.38 & 11.92 & 7.68 \\
          &       &       & $\A=0.4$ & 32.45 & 11.84 & 7.30 \\
          &       &       & $\A=0.5$ & 31.79 & 11.79 & 7.13 \\
          &       &       & $\A=1.0$ & 30.64 & 11.64 & 6.99 \\
    \midrule
    \midrule
    Jan 2, 2015 & $\sim$ & Aug 31, 2017 & MLE  & 17.85 & 13.75 & 15.87 \\
          &       &       & $\A=0.1$ & 20.30 & 13.01 & 12.24 \\
          &       &       & $\A=0.2$ & 20.21 & 12.60 & 10.45 \\
          &       &       & $\A=0.3$ & 20.25 & 12.35 & 9.28 \\
          &       &       & $\A=0.4$ & 20.01 & 12.20 & 8.61 \\
          &       &       & $\A=0.5$ & 19.68 & 12.12 & 8.26 \\
          &       &       & $\A=1.0$ & 19.65 & 11.93 & 7.93 \\
    \bottomrule
    \end{tabular}%
    }
  \label{sub_estimates}%
\end{table}%
The forecasting errors and the number of the observations included in 95\% PIs  are presented in Table \ref{fore_error}. The values in the left sub-table are obtained using the \OU$ $ and the data from Jan 2, 2015. On the other hand, as aforementioned, the one-step-ahead forecasts for the right sub-table are calculated using the data after Mar 2, 2016. The results in Table \ref{fore_error} show that the \OUchg$ $ outperforms the \OU. All the values of RMSE and RMSPE in the right sub-table are less than the corresponding values in the left sub-table. The \OU$ $ estimated by the MLE is shown to yield worst performance and the \OUchg$ $ estimated by the MDPDE with $\A=0.1$ and $\A=0.2$ show best performances in terms for RMSE and RMSPE, respectively. It is important to note that the \OUchg$ $ estimated by the MLE is superior to the \OU$ $ by the MDPDE, implying that the improvement of the forecasting performance by considering parameter changes is greater than by just using the robust estimator. Even though the predicted values are calculated depending only on the drift parameter estimate $(\hat\lambda, \hat\mu)$, the forecasting results above strongly indicate that the \OUchg$ $ is better fitted to the data.

\begin{table}[t]
  \centering
    \tabcolsep=7pt
    \caption{\small 1-step-ahead forecasting errors for the models without and with change-points}\vspace{0.1cm}
  {\footnotesize
    \begin{tabular}{ccccccccc}
    \multicolumn{4}{c}{OU process without parameter change} &       & \multicolumn{4}{c}{OU process with parameter changes} \\
\cmidrule{1-4}\cmidrule{6-9}    MDPDE  & RMSE  & RMSPE & \#     &       &  MDPDE & RMSE  & RMSPE & \# \\
\cmidrule{1-4}\cmidrule{6-9}
  MLE  & \textcolor{red}{0.6171} & \textcolor{red}{0.0489} & 78    &       & MLE  & 0.6035 & 0.0475 & 76 \\
   $\A=0.1$   & 0.6092 & 0.0480 & 76    &       & $\A=0.1$  & \textcolor{blue}{0.6014} & 0.0469 & 76 \\
   $\A=0.2$   & 0.6072 & 0.0477 & 76    &       & $\A=0.2$  & 0.6028 & \textcolor{blue}{0.0468} & 75 \\
   $\A=0.3$   & 0.6068 & 0.0475 & 74    &       & $\A=0.3$  & 0.6042 & 0.0468 & 73 \\
   $\A=0.4$   & 0.6070 & 0.0475 & 73    &       & $\A=0.4$  & 0.6051 & 0.0468 & 71 \\
   $\A=0.5$   & 0.6074 & 0.0475 & 73    &       & $\A=0.5$  & 0.6057 & 0.0469 & 70 \\
   $\A=1.0$   & 0.6081 & 0.0475 & 73    &       & $\A=1.0$  & 0.6075 & 0.0470 & 70 \\
\cmidrule{1-4}\cmidrule{6-9}
\multicolumn{9}{r}{\# denotes the number of observations included in 95\% prediction intervals.}
  \end{tabular}%
  }
  \label{fore_error}%
\end{table}%

Figure \ref{pred} displays the predicted values and 95\% PIs of the \OU$ $ with the ML estimates (left) and the \OUchg$ $ with the MDPD estimates by $\A=0.1$ (right). Although the PI of the \OU$ $ includes all the observations, it produces comparatively longer intervals.  The average lengths of 95\% PIs in the left and right sub-figures  are 3.99 and 2.63, respectively. 75 observations (97.4\%) are included in the PIs of the \OUchg, indicating that the process with changes produces reasonable PIs.

Our empirical findings support that the series are partitioned validly by the proposed test. The estimates in each sub-period are significantly different and the forecasting based on the last sub-period shows better performances. Political and economical events or crises often cause deviating observations in financial data and can also lead to structural changes in underlying models. Our analysis as well as simulation study demonstrates that in such situation where data includes seemingly outliers, the proposed test can effectively detect parameter changes that the existing tests may miss, hence improving forecasting performance.

\begin{figure}[!t]
\includegraphics[height=0.5\textwidth,width=1\textwidth]{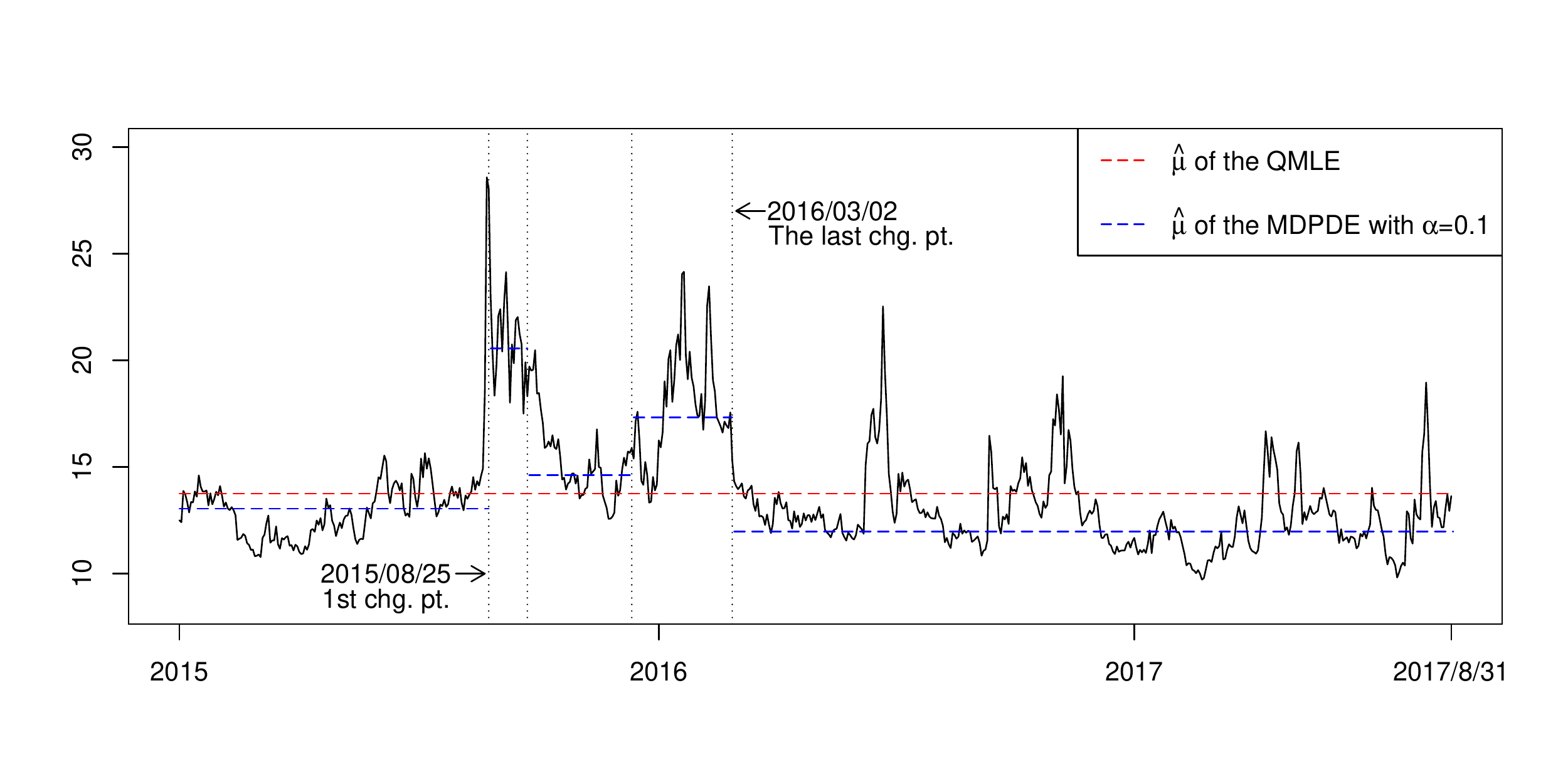}\vspace{-0.9cm}
\caption{\small VKOSPI200 index series up to  Aug 31, 2017 and the estimated change-points}\label{chgpts}
\vspace{1cm}
\includegraphics[height=0.5\textwidth,width=1\textwidth]{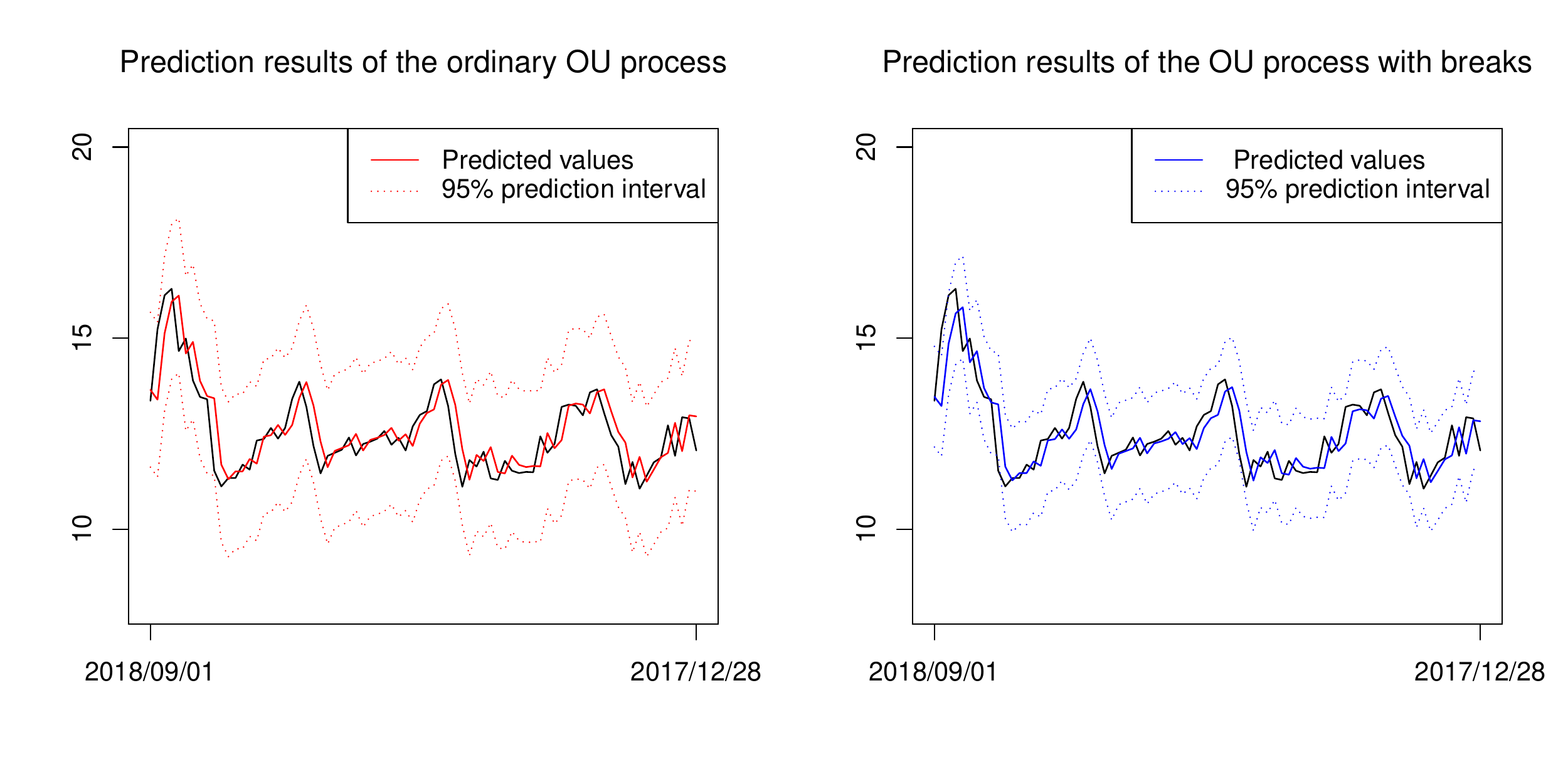}\vspace{-0.9cm}
\caption{\small The predicted values and 95\% prediction intervlas. The ordinary OU process is estimated by the MLE (left) and the process with changes is  by the MDPDE with $\A=0.1$ (right).}\label{pred}
\end{figure}

\section{Concluding remarks}
We have proposed a robust test for dispersion parameter constancy in discretely observed diffusion
processes. The idea used to construct the test is simple and easy to implement: the residuals are calculated using a robust estimate and then a CUSUM test statistics are constructed based on the truncated ones of squared residuals. The limiting null distribution of the proposed test is established and a simulation study demonstrates the promising performance of our test in the presence of outliers. The proposed test possesses a strong robust property against
outliers, whereas the naive CUSUM test is observed to be severely damaged particularly when the sampling interval is short.
Given the situations that high-frequency data have often been obtained, our test will be a good alternative to test for parameter change in such cases.

The extension to general diffusion processes such as $dX_t = a(X_t,\theta) dt+b(X_t,\sigma) dW_t$
is of natural interest. Our results are focused on diffusion processes, but we anticipate that our procedure can be applied to other time series models such as ARMA models and GARCH-type models. Once a robust estimator is given for each model, the same procedure can be adapted to construct test statistics. If one prove the results corresponding to Lemmas \ref{L3} and \ref{L4} below, the same asymptotic result in Theorem \ref{T2} will be obtained under other time series models. We leave these issues as possible topics of future research.

\section{Proofs}
Hereafter, we shall use the
relation $A_n \lesssim B_n$, where $A_n$ and $B_n$ are nonnegative,
to mean that $A_n \leq C B_n$ for some constant $C>0$ and drop the $\alpha$ in $(\hat{\theta}_n^\A, \hat{\sigma}_n^\A)$ and $\hat{Z}_{\A,i}$ for notational simplicity. Further, we denote
\begin{eqnarray*}
Z_i=\frac{1}{\sqrt{h_n}}\big(W_{t_i}-W_{t_{i-1}})\quad and \quad
\Delta_i=\int_{t_{i-1}}^{t_i}\big\{a(X_s,\theta_0)-a(X_{t_{i-1}},\theta_0)\big\}ds.
\end{eqnarray*}
Then, $Z_i$'s are i.i.d. random variables from $N(0,1)$ and, according to Lemma 1 in Lee and Song (2013), we have
\begin{eqnarray}\label{P0}
\max_{1\leq i \leq n} E(|\Delta_i|^k) \lesssim h_n^{1.5k}\quad  \textrm{for any } k\in \mathbb{N}.
\end{eqnarray}

\begin{lm}\label{L1}
Suppose that {\bf A1} and {\bf A3} hold. If $nh_n^2 \rightarrow 0$, then
\begin{eqnarray*}
n^q h_n^m\max_{ i \leq
n}\big|(1+|X_{t_{i-1}}|)^C\,Z_i^k\,\Delta_i^l\big|=o(1)\quad
\textit{a.s.},
\end{eqnarray*}
where $k,l \in \{0,1,2,\cdots\}$ and $m>-1.5l+2q$.
\end{lm}
\begin{proof}
In view of {\bf A3} and (\ref{P0}), we have that for any
$\epsilon, d>0$,
\begin{eqnarray*}
\sum_{n=1}^\infty P\Big( n^q h_n^m\max_{ i \leq n}\big|(1+|X_{t_{i-1}}|)^C Z_i^k\Delta_i^l\big|>\epsilon\Big)
&\leq& \sum_{n=1}^\infty \frac{1}{\epsilon^d}n^{1+qd} h_n^{m d} \max_{i \leq n} \E\big|(1+|X_{t_{i-1}}|)^C Z_i^k\Delta_i^l\big|^d \\
&\les&  \sum_{n=1}^\infty n^{1+qd}h_n^{(1.5l+m)d}=\sum_{n=1}^\infty
o\big(n^{1-(1.5l+m-2q)d/2}\big).
\end{eqnarray*}
Since $1.5l+m-2q>0$, the lemma is yielded by choosing $d$ such that $1-(1.5l+m-2q)d/2 < -1$.
\end{proof}

\begin{lm}\label{L2}
Suppose that {\bf A1} - {\bf A3} hold and let $f(x)$ and $g_n(x)$ belong to $\mathcal{P}$. Assume further that
 $\pa_x f$ exists and belongs to $\mathcal{P}$; $g_n$ converges almost surely to $g \in \mathcal{P}$ and is dominated by some
function in $\mathcal{P}$. If $nh_n^2 \rightarrow 0$, then
\begin{eqnarray*}
\left| \frac{1}{n} \sum_{i=1}^ng_n(Z_i)
f(X_{t_{i-1}}) - E g(Z) \int f(x) \mu_0
(dx) \right| = o(1)\ \ a.s..
\end{eqnarray*}
\end{lm}
\begin{proof}
In view of ergodic property, we get
\begin{eqnarray}\label{L2.1}
\frac{1}{nh_n} \int_0^{nh_n} f(X_s) ds
\stackrel{a.s.}{\longrightarrow} \int f(x) \mu_0 (dx)\ \
\textrm{as} \ n\rightarrow \infty.
\end{eqnarray}
Using Jensen's inequality, Cauchy's inequality and $E\big|X_t-X_{t_{i-1}}\big|^k  \les h_n^{k/2}$ (cf. Kessler (1997)),
we have that for $r>0$
\begin{eqnarray*}
&&E \Big|\frac{1}{n}\sum_{i=1}^n
\frac{1}{h_n}\int_{t_{i-1}}^{t_i}
\big\{f(X_{t_{i-1}})-f(X_s) \big\}ds\Big|^{2r}\\
&& \leq \frac{1}{nh_n}\sum_{i=1}^n\int_{t_{i-1}}^{t_i}
\big\{E(X_{t_{i-1}}-X_s)^{4r}\big\}^{\frac{1}{2}}\Big[\int_0^1 E
\big\{\pa_x f (X_s +u(X_{t_{i-1}}-X_s))\big\}^{4r}du
\Big]^{\frac{1}{2}}ds\\
&&=O(h_n^r)=o(n^{-r/2}).
\end{eqnarray*}
Hence, we have for any $\epsilon>0$ and $r>2$,
\begin{eqnarray*}
\sum_{n=1}^\infty P \Big(\Big|\frac{1}{n}\sum_{i=1}^n
\frac{1}{h_n}\int_{t_{i-1}}^{t_i}
\big\{f(X_{t_{i-1}})-f(X_s) \big\}ds\Big|>\epsilon\Big) \les
\sum_{n=1}^\infty o(n^{-r/2}) <\infty,
\end{eqnarray*}
which together with (\ref{L2.1}) asserts
\begin{eqnarray}\label{L2.2}
 \frac{1}{n} \sum_{i=1}^n
f(X_{t_{i-1}}) \stackrel{a.s.}{\longrightarrow} \int f(x) \mu_0 (dx)\ \ \textrm{as} \ n\rightarrow \infty.
\end{eqnarray}
Next, let
\[ S_i:=\big\{g_n(Z_i)-E[g_n(Z_1)]\big\}f(X_{t_{i-1}}).\]
Then,  $\{S_i\}_{i=1}^n$ forms a martingale difference with respect to $\mathcal{G}_i=\sigma\{W_s : s \leq t_i\}$. Thus, it follows from Burkholder's inequality and Jensen's
inequality that for any $r>1$
\begin{eqnarray*}
E \Big|\frac{1}{n}\sum_{i=1}^n S_i\Big|^{2r} \leq
\frac{1}{n^{r+1}}\sum_{i=1}^n
E|S_i|^{2r}=O\big(n^{-r}\big),
\end{eqnarray*}
and consequently one can see that
\begin{eqnarray*}
\Big|\frac{1}{n}\sum_{i=1}^n g_n(Z_i) f(X_{t_{i-1}}) -
E[g_n(Z_1)]\frac{1}{n}\sum_{i=1}^n f(X_{t_{i-1}})\Big|=o(1)\ \ a.s.
\end{eqnarray*}
Since $E[g_n(Z_1)]$ converges to $E[g(Z_1)]$ by the dominated convergence theorem, the lemma is asserted from (\ref{L2.2}).
\end{proof}

\begin{lm}\label{L3}
Suppose that {\bf A0}-{\bf A4} hold. For a subset $S \subset \mathbb{R}$ and any monotone sequence of subsets $\{S_n\}$ with  $\lim_{n} S_n =S$,
if $nh_n^2\rightarrow 0$, then
\begin{eqnarray*}
\frac{1}{\sqrt{n}}\max_{1\leq k\leq n } \Big|\sum_{i=1}^k1_{S_n}(Z_i^2) (Z_i^2-\hat{Z}_i^2) -\frac{k}{n}\sum_{i=1}^n 1_{S_n}(Z_i^2)(Z_i^2-\hat{Z}_i^2) \Big|=o_P(1),
\end{eqnarray*}
where $1_A$ denotes the indicator function.
 \end{lm}
\begin{proof}
Since
\[\hat{Z}_i=\frac{\sigma_0}{\hat{\sigma}_n}Z_i+\frac{1}{\hat{\sigma}_n}(a_{i-1}(\theta_0)-a_{i-1}(\hat{\theta}_n))\sqrt{h_n}
+\frac{\Delta_i}{\hat{\sigma}_n \sqrt{h_n}},\]
we can express that
\begin{eqnarray}\label{L3.1}
\hat{Z}_i^2-Z_i^2&=&\Big(\frac{\sigma^2_0}{\hat{\sigma}^2_n}-1\Big)Z_i^2+J_{1,i}+J_{2,i}+J_{3,i},
\end{eqnarray}
where
\begin{eqnarray*}
J_{1,i}&=&2\frac{\sigma_0}{\hat{\sigma}^2_n}(a_{i-1}(\theta_0)-a_{i-1}(\hat{\theta}_n))Z_i\sqrt{h_n}, \\
J_{2,i}&=&\frac{1}{\hat{\sigma}^2_n}\Big[(a_{i-1}(\theta_0)-a_{i-1}(\hat{\theta}_n)\sqrt{h_n}+\frac{\Delta_i}{\sqrt{h_n}}  \Big]^2,\\
J_{3,i}&=& 2\frac{\sigma_0}{\hat{\sigma}^2_n} \frac{\Delta_iZ_i}{\sqrt{h_n}}.
\end{eqnarray*}
\noindent
First, note that
\begin{eqnarray}\label{L3.2}
&&\frac{1}{\sqrt{n}}\max_{1\leq k\leq n } \Big|\sum_{i=1}^k  1_{S_n}(Z_i^2)\Big(\frac{\sigma_0^2}{\hat{\sigma}^2_n}-1\Big)Z_i^2 -\frac{k}{n}\sum_{i=1}^n 1_{S_n}(Z_i^2)\Big(\frac{\sigma^2_0}{\hat{\sigma}^2_n}-1\Big)Z_i^2 \Big| \\
& \lesssim &\sqrt{n} |\hat{\sigma}_n -\sigma_0| \max_{1\leq k\leq n } \frac{k}{n}\Big| \frac{1}{k}\sum_{i=1}^k 1_{S_n}(Z_i^2)Z_i^2-\frac{1}{n}\sum_{i=1}^n 1_{S_n}(Z_i^2)Z_i^2  \Big|. \nonumber
\end{eqnarray}
Since $\frac{1}{n}\sum_{i=1}^n 1_{S_n}(Z_i^2)Z_i^2$ converges almost surely by Lemma \ref{L2}, we can obtain that
\begin{eqnarray*}
&&\max_{1\leq k\leq \sqrt{n} } \frac{k}{n}\Big| \frac{1}{k}\sum_{i=1}^k 1_{S_k}(Z_i^2)Z_i^2-\frac{1}{n}\sum_{i=1}^n 1_{S_n}(Z_i^2)Z_i^2  \Big|=o(1)\ a.s.,\\
&&\max_{\sqrt{n}\leq k\leq n }\Big| \frac{1}{k}\sum_{i=1}^k 1_{S_k}(Z_i^2)Z_i^2-\frac{1}{n}\sum_{i=1}^n 1_{S_n}(Z_i^2)Z_i^2  \Big|=o(1)\ a.s.
\end{eqnarray*}
which together with $\sqrt{n} |\hat{\sigma}_n -\sigma_0|=O_P(1)$ yields that(\ref{L3.2}) is $o_P(1)$. In a similar fashion, by using
$\sqrt{nh_n} ||\hat{\theta}_n -\theta_0||=O_P(1)$,
one can show that
\begin{eqnarray*}
\frac{1}{\sqrt{n}}\max_{1\leq k\leq n } \Big|\sum_{i=1}^k  1_{S_n}(Z_i^2)J_{1,i} -\frac{k}{n}\sum_{i=1}^n 1_{S_n}(Z_i^2)J_{1,i} \Big| =o_P(1).
\end{eqnarray*}

Next, to show that the remaining terms in (\ref{L3.1}) are negligible, we note that
\[\frac{1}{\sqrt{n}}\max_{1\leq k\leq n } \Big|\sum_{i=1}^k  1_{S_n}(Z_i^2)(J_{2,i}+J_{3,i}) -\frac{k}{n}\sum_{i=1}^n 1_{S_n}(Z_i^2)(J_{2,i}+J_{3,i}) \Big|
\lesssim \frac{1}{\sqrt{n}}  \sum_{i=1}^n \big|J_{2,i}\big|+\frac{1}{\sqrt{n}}  \sum_{i=1}^n \big|J_{3,i}\big|.\]
Using $\sqrt{nh_n}||\hat{\theta}_n-\theta_0||=O_P(1)$ and Lemma \ref{L1}, we have
\begin{eqnarray}\label{L3.3}
\frac{1}{\sqrt{n}}  \sum_{i=1}^n \big|J_{2,i}\big| \lesssim nh_n||\hat{\theta}_n-\theta_0||^2\frac{1}{\sqrt{n}}\max_{1\leq k\leq n }(1+|X_{i-1}|)^C
+\frac{\sqrt{n}}{h_n} \max_{1\leq k\leq n } \Delta_i^2=o_P(1).
%\big|J_{2,i}\big|\lesssim ||\hat{\theta}_n-\theta_0||^2(1+|X_{i-1}|)^C h_n+\frac{\Delta_i^2}{h_n}.
\end{eqnarray}
Also, it follows from (\ref{P0}) that
\begin{eqnarray*}
\frac{1}{\sqrt{n}}  \sum_{i=1}^n E\big|J_{3,i}\big| \lesssim \frac{1}{\sqrt{nh_n}}\sum_{i=1}^n\sqrt{E\Delta_i^2 EZ_i^2} = O(\sqrt{n}h_n),
\end{eqnarray*}
which implies
\begin{eqnarray*}
\frac{1}{\sqrt{n}}  \sum_{i=1}^n \big|J_{3,i}\big|=o_P(1).
\end{eqnarray*}
This completes the proof.
\end{proof}

\begin{lm}\label{L4}
Suppose that {\bf A0}-{\bf A4} hold. For any $q < 0.5$, if $nh_n^2\rightarrow 0$, then
\begin{eqnarray*}
n^q\max_{1\leq k\leq n } \big|Z_i^2-\hat{Z}_i^2\big|=o_P(1).
\end{eqnarray*}
\end{lm}
\begin{proof}
Following the similar arguments in the proof of (\ref{L3.3}), the lemma can be obtained and thus we omit its proof.
%From (\ref{L3.1}), we have that
%\begin{eqnarray*}
%n^q\max_{1\leq i \leq n}|\hat{Z}_i^2-Z_i^2|&\lesssim& \sqrt{n}|\hat{\sigma}_n -\sigma_0\big| n^{q-1/2}\max_{1\leq i \leq n} Z_i^2+n^q\max_{1\leq i \leq %n}|J_i|
%\end{eqnarray*}
\end{proof}

\begin{lm}\label{L5}
Suppose that {\bf A0}-{\bf A4} hold. If $nh_n^2\rightarrow 0$, then
\begin{eqnarray*}
\hat{\tau}_{j,n}=\frac{1}{n}\sum_{i=1}^nf_{j,M}^2(\hat{Z}_{i}^2)-\big(\frac{1}{n}\sum_{i=1}^nf_{j,M}(\hat{Z}_{i}^2)\big)^2 \stackrel{P}{\longrightarrow} Var (f_{j,M}(Z_1^2) ).
\end{eqnarray*}
\end{lm}
\begin{proof}
Note that $|f_{j,M}(x)-f_{j,M}(y)| \leq |x-y|$ for $\forall x,y>0$. Then, we have
\[ \max\Big\{ \frac{1}{n}\sum_{i=1}^n |f_{j,M}(Z_i^2)-f_{j,M}(\hat{Z}_i^2)|,  \frac{1}{n}\sum_{i=1}^n |f_{j,M}^2(Z_i^2)-f_{j,M}^2(\hat{Z}_i^2)|\Big\} \lesssim \max_{1\leq i\leq n} |Z_i^2 -\hat{Z}_i^2|,\]
which together with Lemma \ref{L2} and \ref{L4} yields the lemma.
\end{proof}

\noindent{\bf Proof of Theorem \ref{T2}}\\
In this proof, we will only deal with the case of $f_{2,M}$ because the case of $f_{1,M}$ can be verified following essentially the same arguments below.

Since $f_{2,M}(Z_1^2), \cdots, f_{2,M}(Z_n^2)$ are i.i.d. random variables, it follows from
the invariance principle and the mapping theorem that
\begin{eqnarray*}
\frac{1}{\sqrt{n}\tau} \max_{1\leq k\leq n } \Big|\sum_{i=1}^k f_{2,M}(Z_i^2) -\frac{k}{n}\sum_{i=1}^nf_{2,M}(Z_i^2)\Big|
\stackrel{d}{\longrightarrow}\sup_{0\leq t \leq 1} |W^0_t|,
\end{eqnarray*}
where $\tau$ denotes the variance of $f_{2,M}(Z_1^2)$. It is therefore sufficient to show that
\begin{eqnarray}\label{T2.1}
\frac{1}{\sqrt{n}}\max_{1\leq k\leq n } \Big|\sum_{i=1}^k H_{i,M} -\frac{k}{n}\sum_{i=1}^n H_{i,M} \Big|=o_P(1),
\end{eqnarray}
where $H_{i,M}=f_{2,M}(Z_i^2)-f_{2,M}(\hat{Z}_i^2)$.

\noindent Let $\epsilon_n=\max_{1\leq k\leq n } |Z_i^2-\hat{Z}_i^2|$, $S_n=[0,2M+1/n^q]$ for some $q\in(0,0.5)$ and
\[I_{n,k}=\sum_{i=1}^k 1_{S_n}(Z_i^2) (Z_i^2-\hat{Z}_i^2) -\frac{k}{n}\sum_{i=1}^n 1_{S_n}(Z_i^2)(Z_i^2-\hat{Z}_i^2).\]
Then, due to Lemma \ref{L3} and \ref{L4}, we have that
 for any $\epsilon>0$,
\begin{eqnarray*}
&&P\Big( \frac{1}{\sqrt{n}}\max_{1\leq k\leq n } \Big|\sum_{i=1}^k H_{i,M} -\frac{k}{n}\sum_{i=1}^n H_{i,M} \Big|>\epsilon\Big)\\
&\leq&
P\Big(\epsilon_n\geq\frac{1}{n^q} \Big)
+ P\Big(\frac{1}{\sqrt{n}}\max_{1\leq k\leq n }\big| I_{n,k}\big|>\frac{\epsilon}{2} \Big)+
P\Big(\epsilon_n<\frac{1}{n^q}, \frac{1}{\sqrt{n}}\max_{1\leq k\leq n } \Big|I_{n,k}-\sum_{i=1}^k H_{i,M} +\frac{k}{n}\sum_{i=1}^n H_{i,M}\Big| >\frac{\epsilon}{2} \Big)\\
&\leq& P\Big(\epsilon_n<\frac{1}{n^q}, \frac{1}{\sqrt{n}}\max_{1\leq k\leq n } \Big|I_{n,k}-\sum_{i=1}^k H_{i,M} +\frac{k}{n}\sum_{i=1}^n H_{i,M}\Big| >\frac{\epsilon}{2} \Big)+o(1).
\end{eqnarray*}
Observe that on $(n^q\epsilon_n<1)$,
\begin{eqnarray*}
\left\{\begin{array}{ll}
H_{i,M}=Z_i^2-\hat{Z}_i^2 &, if\ Z_i^2 \in \Lambda_{1,n}:=[0,M-1/n^q]\\
H_{i,M}=\hat{Z}_i^2-Z_i^2 &, if\ Z_i^2 \in \Lambda_{2,n}:=[M+1/n^q,2M-1/n^q]\\
H_{i,M}\leq |Z_i^2-\hat{Z}_i^2| &, if\ Z_i^2 \in \Lambda_{3,n}:=S_n-(\Lambda_{1,n}\cup \Lambda_{2,n})\\
H_{i,M}=0 &, if\ Z_i^2\in S_n^c.
\end{array}
\right.
\end{eqnarray*}
%where $n$ is enough large such that each subset is not overlapped.
Then, we have that on $(n^q\epsilon_n<1)$,
\begin{eqnarray*}
&&\Big|I_{n,k}-\sum_{i=1}^k H_{i,M} +\frac{k}{n}\sum_{i=1}^n H_{i,M}\Big|\\
&\leq& 2\Big|\sum_{i=1}^k 1_{\Lambda_{2,n}}(Z_i^2) (Z_i^2-\hat{Z}_i^2) -\frac{k}{n}\sum_{i=1}^n 1_{\Lambda_{2,n}}(Z_i^2)(Z_i^2-\hat{Z}_i^2)\Big|
+2\sum_{i=1}^n 1_{\Lambda_{3,n}}(Z_i^2)\big|Z_i^2-\hat{Z}_i^2- H_{i,M}\big|.
\end{eqnarray*}
Since the first term of the RHS above converges to zero in probability by Lemma \ref{L3}, the theorem is established  if we verify that
\begin{eqnarray}\label{P4}
\frac{1}{\sqrt{n}}\sum_{i=1}^n 1_{\Lambda_{3,n}}(Z_i^2)\big|Z_i^2-\hat{Z}_i^2- H_{i,M}\big|=o_P(1).
\end{eqnarray}
To prove (\ref{P4}), denote by $n_B$ the number of $Z_i^2$'s belonging to $\Lambda_{3,n}$ and let $\delta_n=P(Z_1^2 \in \Lambda_{3,n})$. Then,
$n_B$ becomes a random variable from $B(n,\delta_n)$. Using the fact that $\sqrt{a+x}-\sqrt{a-x} =O(x/\sqrt{a})$ as $x\rightarrow 0$, it can be readily seen that
\begin{eqnarray}\label{P5}
\delta_n \lesssim 2\Big(\frac{1}{\sqrt{M}}+\frac{1}{\sqrt{2M}}\Big)\frac{1}{n^q}.
\end{eqnarray}
Now, take a triangular array of random varibles $\{X_{ni}| n\geq 1, 1\leq i \leq n\}$ defined on a new probability space $(\Omega', \mathcal{F}', P')$ such that $X_{n1}, X_{n2},\cdots, X_{nn}$ are i.i.d. Bernoulli random variables  with success probability $\delta_n$, which is possible due to Theorem 5.3 of Billingsley (1995).
Letting $r_n=n_B/n-\delta_n$, it follows from the law of the iterated logarithm that
\begin{eqnarray*}
%P\Big(\frac{n}{\sqrt{\delta_n(1-\delta_n)}}|r_n| \geq \sqrt{2n\log\log n} \Big) =P_1\Big(\Big|\sum_{i=1}^n \frac{X_{ni}-\delta_n}{\sqrt{\delta_n(1-\delta_n)}}\Big|\geq \sqrt{2n\log\log n}\Big)=o(1).\\
P\Big(\overline{\lim_{n}} \frac{1}{\sqrt{2n\log\log n}}\frac{n|r_n|}{\sqrt{\delta_n(1-\delta_n)}}>1\Big) =P'\Big(\overline{\lim_{n}} \frac{1}{\sqrt{2n\log\log n}}\Big|\sum_{i=1}^n \frac{X_{ni}-\delta_n}{\sqrt{\delta_n(1-\delta_n)}}\Big|>1\Big)=0,
%P\Big(\overline{\lim_{n}} \frac{\sqrt{n}|r_n|}{\sqrt{\delta_n(1-\delta_n)}}>1\Big) =P_1\Big(\overline{\lim_{n}} \frac{1}{\sqrt{n}}\Big|\sum_{i=1}^n \frac{X_{ni}-\delta_n}{\sqrt{\delta_n(1-\delta_n)}}\Big|>1\Big)=0,
\end{eqnarray*}
and thus, by (\ref{P5}), we have
\[ r_n=O\big( n^{q/2-1/2}\sqrt{\log\log n}\big)\quad a.s.\]
Hence, in view of Lemma \ref{L4}, we have that
\begin{eqnarray*}
\frac{1}{\sqrt{n}} \sum_{i=1}^n 1_{\Lambda_{3,n}}(Z_i^2)\big|(Z_i^2-\hat{Z}_i^2) - H_{i,M}\big| \leq 2\frac{n_B}{\sqrt{n}}\epsilon_n
&\lesssim& \sqrt{n}r_n\epsilon_n +\sqrt{n}\delta_n\epsilon_n\\
&\lesssim& \frac{\sqrt{\log\log n}}{n^{q/2}} n^{q} \epsilon_n +n^{1/2-q}\epsilon_n=o_P(1),
\end{eqnarray*}
which establish (\ref{P4}), and thus the proof is completed.
\hfill{$\Box$}\vspace{0.2cm}\\

\noindent{\bf Acknowledgments}\\
This research was supported by Basic Science Research Program through the National Research Foundation of Korea (NRF) funded by the Ministry of Science, ICT and Future Planning (NRF-2016R1C1B1015963).

\end{document}